\documentclass[reqno,pdftex]{amsart}
\usepackage{amsmath,amsthm, amssymb,verbatim,bm}
\usepackage[hyphens]{url}
\usepackage{amscd}
\usepackage[all]{xy}

\numberwithin{equation}{section}

\theoremstyle{plain}
  \newtheorem{lem}{Lemma}[section]
  \newtheorem{prop}[lem]{Proposition}
  \newtheorem{thm}[lem]{Theorem}

\theoremstyle{definition}
  \newtheorem{defi}[lem]{Definition}
  \theoremstyle{remark}

  \newtheorem{rem}{Remark}
  
\renewcommand{\theenumi}{\roman{enumi}}
\renewcommand{\labelenumi}{(\theenumi)}
\newcommand{\If}{\infty}
\newcommand{\op}{\operatorname}
\newcommand{\lra}{\longrightarrow}
\newcommand{\xra}{\xrightarrow}
\renewcommand{\ker}{\operatorname{Ker}}

\newcommand{\aut}{\operatorname{Aut}}
\newcommand{\ee}{\epsilon}
\renewcommand{\hom}{\operatorname{Hom}}

\newcommand{\BZN}{B\Bbb{Z}/p^{[2p^n-1]}}
\newcommand{\BZ}{B(\Bbb{Z}/p)}
\newcommand{\CP}{\Bbb{C}P}
\newcommand{\CPIf}{\Bbb{C}P^{\infty}}
\newcommand{\CPN}{\Bbb{C}P^{p^n-1}}
\newcommand{\Ftwo}{\Bbb{F}_2}
\newcommand{\FP}{\Bbb{F}_p}

\newcommand{\OT}{\otimes}
\newcommand{\HOT}{\hat{\otimes}}
\newcommand{\WH}{\widehat}
\newcommand{\invlim}{\varprojlim}

\newcommand{\KN}{K(n)_*}

\begin{document}
\title[$K(n)_* (K(n)) $ and
the fiber product of  automorphism groups]{The stable cooperations
  of Morava $K$-Theory and
the fiber product of  automorphism groups of formal group laws}
\author[M.Inoue]{Masateru INOUE}
\address{Department of Applied Mathematics,
  Faculty of Science,
  Okayama University of Science,
  1-1 Ridaicho, Kita-ku, Okayama-shi 700-0005, Japan}
\email{inoue@ous.ac.jp}
\date{\today}
\maketitle

\begin{abstract}
   There are many previous studies on the Hopf algebra $K(n)_*(K(n))$,
  the stable cooperations 
  of $n$th Morava $K$-theory at an odd prime.
  Whereas the main part of $K(n)_*(K(n))$  corepresents
  the group-valued functor consisting of  strict automorphisms 
  of the Honda  formal group law of height $n$,
  relations between
    the whole structure of $K(n)_*(K(n))$ including the exterior part
  and formal group laws   have not been investigated well.
  Firstly, we constitute a functor $C(-)$ which is given by
  the fiber product of
  two natural homomorphism between
 subgroups of  automorphisms  of   formal group laws,
  and the Hopf algebra $C_*$  corepresenting $C(-)$.
  Next, we  construct a Hopf algebra homomorphism $\kappa^*:C_*\to K(n)_*(K(n))$
  naturally.
  To relate $C_*$ to $K(n)_*(K(n))$, we use  stable comodule algebras
  which are introduced  by Boardman.
  From the algebra structure of $K(n)_*(K(n))$ which is given
  by W\"urgler and Yagita,   we see that $\kappa^*$ is an isomorphism.
  Since we formulate  $C_*$ by using  formal group laws,
  the isomorphism $\kappa^*$ clarifies  relationship between
  the Hopf algebra structure of $K(n)_*(K(n))$ including
  the exterior algebra part and formal group laws.
\end{abstract}

\section{Introduction}
Throughout this paper, we fix an odd prime $p$.
Let $K(n)_*(K(n))$ be the stable cooperations
of $n$th Morava $K$-theory $K(n)$ at $p$.
For a spectrum $E$ and a space $X$, we assign the degree $n$
to  cohomology classes in  $E^n(X)$.
Hence homology  classes in $E_n(X)$ are of degree $-n$.
The coefficient ring of $K(n)$ is the graded field $K(n)_*= \FP [v_n,v_n^{-1}] $,
where $\deg v_n=-2(p^n-1)$.
Any algebra with grading is $\Bbb{Z}$-graded commutative.
We denote  tensor products and  completed tensor products over $K(n)_*$
by  $\OT$ and $\hat{\OT}$ respectively, except in Section \ref{sec:boardman}

Our aim is to describe the Hopf algebra structure of
$K(n)_*(K(n))$ in terms of formal group laws in a similar way to \cite{I,I2}.
W\"urgler \cite{W2} and Yagita \cite{Y1}  determined the Hopf algebra $K(n)_*(K(n))$.
As a $K(n)_*$-algebra, we have
\begin{equation}
  \label{eq:kncoop}
K(n)_*(K(n))\cong
\Sigma(n)\OT
E(\tau_0,\tau_1,\dotsc , \tau_{n-1}),\quad
\Sigma(n)=K(n)_*[t_1,t_2,\dotsc ]/(v_nt_i^{p^n}-v_n^{p^i}t_i)  
\end{equation}
with $\deg \tau_i=-(2p^i-1),\deg t_i=-2(p^i-1)$.
Here $E$ is an exterior algebra.
The comultiplication of $\tau_i$ is given by
\begin{align*}
  \Delta(\tau_i)=1\OT \tau_i+\sum_{j=0}^i\tau_j\OT t_{i-j}^{p^j}.
\end{align*}
Since there is an isomorphism $$ \Sigma(n)\cong
K(n)_*\OT_{BP_*} BP_*(BP)\OT_{BP_*} K(n)_*,\ 
$$
the comultiplication
on $\Sigma(n)$
is inherited from the stable cooperations 
$BP_*(BP)$ of  the Brown-Peterson spectrum $BP$.
See \cite{adams,R} for more details on $BP$.
Let $H_n(x,y)$ be the Honda formal group law \cite{Honda}
whose $p$-series is given by $[p](x)=v_nx^{p^n}$, and $\aut_{H_n}(R_*)$ the group of  strict automorphisms of $H_n$ over a $K(n)_*$-algebra $R_*$.
As is well known,
the Hopf algebra $\Sigma(n)$ corepresents the functor $\aut_{H_n}(-)$.
That is, there is an isomorphism
$\aut_{H_n}(R_*)\cong \hom_{K(n)_*\text{-alg.}}(\Sigma (n),R_*) $
of groups.
See \cite{BW,R}.

Previous studies on the exterior algebra  part
$E(\tau_0,\tau_1,\dotsc , \tau_{n-1})$
are known as follows.
Yagita \cite{Y1}  determined the Hopf algebra  $K(n)_*(K(n))$ and
the comultiplication on $E(\tau_0,\tau_1,\dotsc , \tau_{n-1})$
by using the reduction map from the stable cooperations $k(n)_*(k(n))$ of
the connective Morava $K$-theory to the mod $p$ dual Steenrod algebra.
W\"urgler \cite{W2} also determined the Hopf algebra $K(n)_*(K(n))$
in a different way.
Baker \cite{B} focused on  $A_{\If}$-structures of
the $I_n$-adic completion $\widehat{E(n)}$ of the Johnson-Wilson spectrum
$E(n)$, and showed that
there is an inverse system of $A_{\If}$-module spectra over $\WH{E(n)}$
$$\cdots \lra E(n)/I_n^{k+1}\lra E(n)/I_n^{k}\lra \cdots \lra E(n)/I_n^{2}\lra
E(n)/I_n=K(n)$$
such that $\op{holim}_kE(n)/I_n^{k}\simeq \WH{E(n)}$.
Here we fix an  $A_{\If}$-structure of $K(n)$.
Baker and W\"urgler \cite{BW} constructed
a cofiber sequence of $A_{\If}$-module spectra over $\widehat{E(n)}$
$$\bigvee_{0\leq k\leq n-1}\Sigma^{2p^k-2}K(n)\lra E(n)/I_n^2\lra E(n)/I_n=K(n), $$
and the Bockstein operations
$\mathcal{Q}_k:K(n)\to \Sigma^{2p^k-1}K(n)$ of the cofiber sequence.
Let $\BZN$ be the $(2p^n-1)$-skeleton of the classifying space $B\Bbb{Z}/p$.
Given the canonical generators $\ee, x$ with $\deg \ee=1, \deg x=2$, we have
$$K(n)^*(\BZN)\cong E(\ee )\OT K(n)^*[x]/(x^{p^n}).   $$
The operations $\mathcal{Q}_k$ are characterized by the following properties:
\begin{enumerate}
\item For all $0\leq k\leq n-1$, the cohomology operation  
  $$\mathcal{Q}_k:K(n)^*(-)\lra K(n)^*(-) $$
    is a $K(n)^*$-derivation;
\item $\mathcal{Q}_k(\ee)=x^{p^k}$;
\item $\mathcal{Q}_k(y)=0$ for a generator $y$ of
  $K(n)^*(\CPIf)=K(n)^*[[y]]$.
\end{enumerate}
The dual elements of $\mathcal{Q}_0,\dotsc \mathcal{Q}_{n-1}$
yield   $\tau_0,\dotsc ,\tau_{n-1}$ of $K(n)_*(K(n))$.
In particular, the multiplicative coaction
$\Psi_0$
in \cite[Section 1]{BW}
is similar to $\gamma_{\BZN}$ in Proposition \ref{prop:optoGa}.
The studies seem to be related to the structure of formal group laws.
However, it is not clear how
the whole structure of $K(n)_*(K(n))$ including the exterior algebra part
is connected with formal group laws.
This paper highlights connections between them. 

We proceed as follows.
We construct three group-valued functors
$\aut_{H_n}(-)$, $\aut_{\bar{g}_a}(-)$ and $\aut_{\bar{G}_a}(-)$,
which consist of subgroups of  automorphisms  of the formal group laws
$H_n$, $\bar{g}_a$ and $\bar{G}_a$  in Sections 2, 3 and 4, respectively.
These structures induce natural transformations of groups
\begin{equation*}
\alpha:\aut_{H_n}(-)\lra \aut_{\bar{g}_a}(-),\quad
\beta:\aut_{\bar{G}_a}(-)\lra \aut_{\bar{g}_a}(-).
\end{equation*}
The fiber product $C(-)$ of $\alpha$ and $\beta$
is corepresented by the  Hopf algebra $C_*$, which
is isomorphic to $K(n)_*(K(n))$.
To relate $C_*$ to $K(n)_*(K(n))$,
we need stable comodule algebras which is introduced by Boardman \cite{board}.
This leads to a natural  homomorphism $\kappa^*:C_*\to K(n)_*(K(n))$ of Hopf algebras.
Using  the algebra structure of $K(n)_*(K(n))$ in \cite{W2,Y1},
we see that $\kappa^*$ is an isomorphism.
We do not need the Hopf algebra structure of  $K(n)_*(K(n))$.
We emphasize that $\kappa^*$ is constructed naturally,
although we can not prove that $\kappa^*$ is an isomorphism
in a self-contained manner.
The isomorphism $\kappa^*$ reveals  relationship between $K(n)_*(K(n))$
and  formal group laws.

Firstly, we survey \cite{I}
before explaining the details of this paper. 
We use different notations from \cite{I}.
Let $H$ be the mod 2 Eilenberg-MacLane spectrum. 
We introduce functors $\aut_{g_a}(-)$
and $\op{Op}_{H}(-)$ from the category of
non-negative graded commutative $\Ftwo$-algebras 
to the category of groups,
and a natural transformation $\op{Op}_{H}(-)\to \aut_{g_a}(-)$.
The natural transformation  induces the Hopf algebra structure of the
mod 2 dual Steenrod algebra $H_*(H)$ directly.
Historically, Milnor \cite{mil}  determined $H_*(H)$
by calculating the dual of the Steenrod algebra $H^*(H)$.
Let $R_*$  be a non-negative graded commutative $\Ftwo$-algebra,
and
$$g_a(x,y)=x+y$$
the additive formal group law over $R_*$.
Here $\deg x=\deg y=1$, and elements of $R_n$ have degree $-n$.
A homogeneous formal power series
$f(x)=x+\sum_{i\geq 2}a_ix^i\in R_*[[x]]$ of degree 1 is  
called a {\em strict} automorphism of $g_a$ over $R_*$
if it satisfies $f(x+y)=f(x)+f(y)$.
We denote the  strict automorphisms
of $g_a$ over $R_*$ by $\aut_{g_a}(R_*)$,
whose multiplication 
is defined by $f\cdot g=g\circ f(x)$.
Then
$\aut_{g_a}(-)$ is a group-valued functor.
It is easily seen that
a strict automorphism  has a form
$f(x)=x+\sum_{j\geq 1}b_jx^{2^j} $ with $b_j\in R_{2^j-1}$.
Therefore
$\aut_{g_a}(-)$ is corepresented by the Hopf algebra
$$\bar{A}_*=\Ftwo [\xi_1,\xi_2,\dotsc,\xi_n,\dotsc ]\quad (\deg \xi_n=-(2^n-1))$$
with the comultiplication $\xi_n\mapsto \sum_{0\leq i\leq n}\xi_{n-i}^{2^i}\OT \xi_i$.
In other words,
there is the  natural isomorphism
$$\theta: \hom_{\Ftwo\text{-alg}}(\bar{A}_*,R_*)\cong \aut_{g_a}(R_*)$$
which is  given by $\theta(f)=x+\sum_{j\geq1}f(\xi_j)x^{2^j}$.

Let $X$ be a CW-complex, and
$\gamma_X:H^*(X)\to H^*(X)\HOT R_*$  a natural homomorphism of graded abelian groups. 
Here the degree $k$ part of $H^*(X)\HOT R_*$
is $\prod_{n\geq 0}H^{k+n}(X)\OT_{\Ftwo} R_n $ by abuse of notation.
\begin{defi}[\cite{I},Definition 2.1]\label{defi:mulop0}
A natural homomorphism  $\gamma_X:H^*(X)\to H^*(X)\HOT R_*$ as above
is called a {\em multiplicative operation} over $R_* $
if $\gamma_X$ satisfies the conditions:
\begin{enumerate}
\item The following diagram is commutative:
$$
\xymatrix{
  H^*(X)\OT H^*(Y)\ar[dd]_{\times}\ar[r]^-{\gamma_X \OT\gamma_Y}
  & (H^*(X)\HOT R_*)\OT (H^*(Y)\HOT R_*)\ar[d]  \\
  & (H^*(X)\OT H^*(Y))\HOT (R_*\OT R_*)\ar[d]^{\times \OT \mu}\\
  H^*(X\times Y)\ar[r]^{\gamma_{X\times Y}} & H^{*}(X\times Y)\HOT R_*
},
$$
where $\times$ is the cross product,
$\mu$ is the multiplication on $R_*$, and the unlabeled map 
interchanges the second and third factors.
\item\label{1.1-2} $\gamma_{S^i} (u_i)=u_i\OT 1$ for $i=0,1$.
  Here $S^i$ is  the $i$-sphere with basepoint $o$,  and
  $u_i$ is  the canonical generator of $ {H}^{i}(S^i,o)\subset H^{i}(S^i)$.
  
\end{enumerate}
\end{defi}
We add the case $i=0$ to Condition (\ref{1.1-2}) of \cite[Definition 2.1]{I}.
It follows from Conditions (i) and (ii)
that $\gamma_{S^k}(u_k)=u_k\OT 1$ for any $k\geq 0$,
and  that $\gamma_X $ is a stable operation.
Let $\op{Op}_H(-)$ be the functor which sends $R_*$ to
the set of multiplicative operations over $R_*$.
For $\gamma_X,\ \gamma_X'\in \op{Op}_H(R_*)$,
we define the multiplication $\gamma_X\cdot \gamma_X'$
by
the composition
\begin{equation}\label{eq:iterateH}
\gamma_X\cdot \gamma_X':H^*(X)\xrightarrow{\gamma_X'}
H^*(X)\HOT R_*\xrightarrow{\gamma_X \OT 1}H^*(X)\HOT R_* \HOT R_*
\xra{1\OT \mu} H^*(X)\HOT R_*.
\end{equation}
There exists a {\em universal} multiplicative operation
$$\rho_X:H^*(X)\lra H^*(X)\HOT H_*(H) $$
in the following sense.
Given $\gamma_X \in \op{Op}_H(R_*)$,
there exists a unique algebra homomorphism $\bar{\gamma}:H_*(H)\to R_*$
such that the following diagram is commutative:
$$\xymatrix{
  H^*(X)\ar[r]^-{\rho_X}\ar[rd]_{\gamma_X}
  & H^*(X)\HOT H_*(H)\ar[d]^{1\HOT \bar{\gamma}}\\
& H^*(X)\HOT R_*.
}
$$
Therefore, there is a one-to-one correspondence between the multiplicative
operations over $R_*$ and the algebra homomorphisms $H_*(H)\to R_*$.
That is, we have an isomorphism
$\op{Op}_H(R_*)\cong \hom_{\Ftwo\text{-alg}}(H_*(H),R_*) $ of sets.
Since $\rho_X$ is a comodule, we have
$$
    \xymatrix{
(\rho_{X}\OT 1)\circ\rho_{X}=(1\OT \psi)\circ \rho_{X}:
H^*(X)\hat{}\ar[r]^-{\rho_{X}} &H^*(X)\HOT H_*(H)
\ar@<0.5ex>[r]^-{\rho_{X}\OT 1} \ar@<-0.5ex>[r]_-{1\OT \psi} 
& H^*(X)\HOT H_*(H)\HOT H_*(H),
      }
      $$
where $\psi$ is the comultiplication on $H_*(H)$. 
Moreover the multiplication
$\gamma_{X}  \cdot \gamma'_{X}$ coincides with the composition
$$ H^*(X)\xrightarrow{\rho_X}
H^*(X)\HOT H_*(H)\xrightarrow{\rho_X \OT 1}H^*(X)\HOT H_*(H) \HOT H_*(H)
\xra{1\OT \bar{\gamma}\OT \bar{\gamma}'}H^*(X)\HOT R_* \HOT R_*
\xra{1\OT \mu} H^*(X)\HOT R_*.$$
The inverse element of $\gamma_X$ is defined as
$$ \gamma_X^{-1}=(1\OT \bar{\gamma})\circ(1\OT \chi)\circ\rho_X:
H^*(X)\xrightarrow{\rho_X} H^*(X)\HOT H_*(H)\xrightarrow{1\OT \chi}
H^*(X)\HOT H_*(H)\xrightarrow{1\OT \bar{\gamma}} H^*(X)\HOT R_*,$$
where $\chi$ is the conjugation of $H_*(H)$.
Thus $\op{Op}_H(-)$ is a group-valued functor,
and
there is a natural isomorphism
$$\op{Op}_H(-)\cong \hom_{\Ftwo\text{-alg}}(H_*(H),-) $$
of groups. 
In particular, the universal multiplicative operation $\rho_X$
corresponds to the identity map $1_{H_*(H)}$.
Note that  $\rho_X$ coincides with
Milnor coaction in \cite{mil}.
The universal multiplicative operation $\rho_X$ is almost the same as
the universal coaction of the stable comodule algebra $H^*(X)$
in Definition \ref{defi:comalg} or \cite[Theorem 12.8]{board}.
The author \cite{I}  constructed it by using  spectra, and
Boardman \cite{board}
introduced it by studying various operations in terms of comonads.
In this paper, we employ Boardman's stable comodule algebras.

To relate $\op{Op}_H(-)$ to $\aut_{g_a}(-)$,
we substitute the classifying space $X=B\Bbb{Z}/2$
into  multiplicative operations $\gamma_X$.
For the canonical generator $x$ of $H^1(B\Bbb{Z}/2)$,
we have $H^*(B\Bbb{Z}/2)\cong \Ftwo [x]$.
The $i$th projections $p_i:B\Bbb{Z}/2\times B\Bbb{Z}/2\to B\Bbb{Z}/2$
for $i=1,2$ induce the generators $x_i=p_i^*(x)\in H^1(B\Bbb{Z}/2\times B\Bbb{Z}/2)$,
and we obtain $H^*(B\Bbb{Z}/2\times B\Bbb{Z}/2)\cong \Ftwo [x_1,x_2]$.
The classifying space $B\Bbb{Z}/2$ is an $H$-space with
 the multiplication
$m:B\Bbb{Z}/2\times B\Bbb{Z}/2\to B\Bbb{Z}/2$.
The pullback $m^*(x)=x_1+x_2$ coincides with the additive formal group
law $g_a$.
For $\gamma_X\in \op{Op}_H(R_*)$,
we see that $\gamma_{B\Bbb{Z}/2} (x)$
is an element of the ring 
$R_*[[x]]\cong H^*(B\Bbb{Z}/2)\HOT R_*$.
Condition (i) of Definition \ref{defi:mulop0}
and the naturality of $\gamma$ with respect to $m$
lead to
$$\gamma_{B\Bbb{Z}/2}(x)+\gamma_{B\Bbb{Z}/2}(y)=\gamma_{B\Bbb{Z}/2}(x+y).$$
Condition (ii)
implies $\gamma_{B\Bbb{Z}/2}(x)\equiv x$ modulo $(x^2)$.
These mean that $\gamma_{B\Bbb{Z}/2} (x)$
is a strict automorphism of $g_a$ over $R_*$.
Assigning $\gamma_{B\Bbb{Z}/2} (x)$ to $\gamma_X$,
we have the natural homomorphism of groups 
$$F:\op{Op}_H(-)\cong \hom_{\Ftwo\text{-alg}}(H_*(H),-)\lra \aut_{g_a}(-)
\cong \hom_{\Ftwo\text{-alg}}(\bar{A},-). $$
By Yoneda's lemma, we have the corresponding  homomorphism of Hopf algebras
$$\overline{F(\rho_X)}:\bar{A}_*\lra H_*(H).$$
We show that $\overline{F(\rho_X)}$ is an isomorphism as follows.
We modify M\`ui's results \cite{Mui1,Mui2}, and obtain another multiplicative operations
$$S_n:H^*(X)\lra H^*(X)\hat{\OT} D[n]_*,$$
where $D[n]_*=\Ftwo [\xi[n]_1,\dotsc ,\xi[n]_n]$
with $\deg \xi[n]_i=-(2^i-1) $.
 Lomonaco \cite{L} constructed the same operation.
Since $F$ is a natural transformation,
the following diagram 
is commutative:
$$
\xymatrix{
  \bar{A}_*\ar[rd]_{\overline{F(S_n)}}\ar[r]^{\overline{F(\rho_X)}}& H_*(H)\ar[d]^{\bar{S}_n}\\
  & D[n]_*.
}
$$
It follows from  \cite{Mui1,Mui2}
that $\overline{F(S_n)}:\bar{A}_*\lra D[n]_*$
is given by $\xi_i\mapsto \xi[n]_i$ for $i\leq n$, and
$\xi_i\mapsto  0$ for $i>n$.
Therefore $\overline{F(S_n)}$ is an isomorphism in some  range of
degrees for sufficiently large $n$,
and we see that   $\overline{F(\rho_X)}:\bar{A}_*\to H_*(H)$ is injective.
Serre \cite{serre} determined the mod 2 cohomology of the Eilenberg-MacLane space $K(\Pi ,n)$,
where $\Pi$ is a finitely generated abelian group.
It induces the Poincar\'e series of $H_*(H)$, which coincides with
that of $\bar{A}_*$.
Hence  $\overline{F(\rho_X)}$  is an isomorphism.
The Hopf algebra $\bar{A}_*$ derives from automorphisms of the
additive formal group $g_a$, and we construct $F$ naturally. 
Thus, we determine the Hopf algebra  $H_*(H)$
directly 
without using the Steenrod algebra $H^*(H)$

Next, we recall  \cite{I2},
which extends to the odd primary case. 
Let $H$ be the mod $p$ Eilenberg-MacLane spectrum.
The mod $p$ dual Steenrod algebra is
$$H_*(H)\cong E(\tau_0,\tau_1,\dotsc)\OT \FP[\xi_1,\xi_2,\dotsc ]  $$
whose comultiplication is given by $\Delta(\xi_n)=\sum_{i=0}^n\xi_{n-i}^{p^i}\OT \xi_i$
and $\Delta(\tau_n)=\tau_n\OT 1+\sum_{i=0}^n\xi_{n-i}^{p^i}\OT \tau_i$.
To obtain $H_*(H)$, we  use 
the classifying space $B\Bbb{Z}/p$ and a
graded 2-dimensional additive formal group law $G_a$
instead of $B\Bbb{Z}/2$ and $g_a$, respectively.
Given the canonical generators $\epsilon \in H^1(B\Bbb{Z}/p)$
and $x\in H^2(B\Bbb{Z}/p)$,
we have $H^*(B\Bbb{Z}/p)\cong E(\epsilon)\OT \FP [x]$.
The multiplication $m$ on $B\Bbb{Z}/p$ induces
\begin{align*}
  m^*(\epsilon) =\epsilon_1+\epsilon_2,\ m^*(x)=x_1+x_2\in 
  E(\ee_1,\ee_2)\OT \FP[x_1,x_2]\cong H^*(B\Bbb{Z}/p\times B\Bbb{Z}/p),
\end{align*}
where $\ee_i=p_i^*(\ee),\ x_i=p_i^*(x) $ for the $i$th projections
$p_i:B\Bbb{Z}/p\times B\Bbb{Z}/p \to B\Bbb{Z}/p $. 
We define $G_a$ as the pair
\begin{equation}\label{eq:2dimAFGL}
G_a(X_1;X_2)=G_a(\ee_1,x_1;\ee_2,x_2)=(m^*(\ee), m^*(x))=(\ee_1+\ee_2,x_1+x_2).  
\end{equation}
Let $R_*$ be a non-negative graded $\FP$-algebra.
A pair of homogeneous power series 
\begin{align*}
  f(1)(\epsilon,x)=\ee+\sum_{i=0}^{\infty}a_i x^{p^i},\ 
  f(2)(\epsilon,x)=x+\sum_{i=1}^{\infty}b_i x^{p^i} \in E(\epsilon)\OT R_*[[x]]
\end{align*}
with $\deg f(1)=1,\ \deg f(2)=2$
is  an automorphism of $G_a$ over $R_*$.
In other words, the pair satisfies
\begin{align*}
  f(1)(\ee_1+\ee_2,x_1+x_2)&=f(1)(\ee_1,x_1)+f(1)(\ee_2,x_2)
  \\
   f(2)(\ee_1+\ee_2,x_1+x_2)&=f(2)(\ee_1,x_1)+f(2)(\ee_2,x_2).
\end{align*}
The pair $(f(1), f(2))$ is  called  a
{\em quasi-strict} automorphism of $G_a$ over $R_*$.
If $a_1=0$, then it is said to be {\em strict} in \cite{H}.
Let $\aut_{G_a}(R_*)$ be the group of quasi-strict automorphisms
of $G_a$ over $R_*$.
The remaining parts  proceed in a similar way to \cite{I}.
Cartan \cite{C} determined the mod $p$ cohomology of the
Eilenberg-MacLane space $K(\Pi,n)$, where $\Pi$ is a cyclic group.
The result   and
a natural transformation $\op{Op}_H(-) \to \aut_{G_a}(-)$
lead to the Hopf algebra structure of $H_*(H)$.
The elements $\tau_0,\dotsc ,\tau_{n-1}$ of $H_*(H)$ is similar to
$\tau_0,\dotsc ,\tau_{n-1}$ of $K(n)_*(K(n))$.
This indicates that we have relationship between the exterior part of $K(n)_*(K(n))$ and  formal group laws.

There are previous studies on the functor $\aut_{G_a}(-)$.
Hopkins \cite[Section 7]{hop} wrote the supergroup $\Bbb{G}_a^{1,1}$ for $G_a$ and referred to automorphisms of $\Bbb{G}_a^{1,1}$,
which is the same as $\aut_{G_a}(-)$.
Yamaguchi \cite[Appendix]{yam}  constructed a functor
which is isomorphic to $\aut_{{G}_a}(-)$ as follows.
Let $g_a'(x,y)=x+y$ be the additive formal group law over 
the dual numbers $R_*[\epsilon]/(\epsilon^2)$,
and $\Gamma (R_*)$  the automorphism group
of $g_a'$ over $R_*[\epsilon]/(\epsilon^2)$  with some conditions.
Then we have an isomorphism $\aut_{{G}_a}(-)\cong \Gamma (-)$.
In \cite{I2}, the author 
mainly  used $\Gamma (-)$, and dealt with $\aut_{G_a}(-)$ in Appendix.
Though the author  did not refer to the results in \cite{I2},
the functor $\aut_{{G}_a}(-)$ had been obtained before.

Lastly, we outline this paper in more detail.
In Section \ref{sec:fgls},
we recall the  Honda formal group law $H_n$ of height $n$ and the
group-valued functor $\aut_{H_n}(R_*)$ consisting of the strict automorphisms
of $H_n$ over a $K(n)_*$-algebra $R_*$.
In Section \ref{sec:afgls},
we survey   formal group law chunks, truncated forms of 1-dimensional
formal group laws from  \cite{H}.
We need only the additive formal group law chunk
$\bar{g}_a(x_1,x_2)=x+y\in R_*[x_1,x_2]/(x_1,x_2)^{p^n}$ and
the functor $\aut_{\bar{g}_a}(R_*)$ consisting of the strict automorphism group
of $\bar{g}_a$. 
In Section \ref{sec:2fgl},
we extend $\bar{g}_a$ and $\aut_{\bar{g}_a}(-)$
to a 2-dimensional additive formal group law chunk $\bar{G}_a$
and the quasi-strict automorphism group $\aut_{\bar{G}_a}(R_*)$.
We define
$\bar{G}_a(\ee_1,x_1;\ee_2,x_2)=(\ee_1+\ee_2,x_1+x_2)$
as the pair of elements of $E(\ee_1,\ee_2)\OT R_*[x_1,x_2]/(x_1,x_2)^{p^n}$.
In a similar way to \cite{I2},
the Hopf algebra $B_*$ which corepresents  the functor $\aut_{\bar{G}_a}(-)$
yields the elements $\tau_0,\dotsc ,\tau_{n-1}$ of $K(n)_*(K(n))$.
In Section \ref{sec:fiber},
we construct
natural homomorphisms of groups
$$\alpha:\aut_{H_n}(-)\lra \aut_{\bar{g}_a}(-),\quad
\beta:\aut_{\bar{G}_a}(-)\lra \aut_{\bar{g}_a}(-),$$
and the fiber product $C(-)$ of $\alpha $ and $\beta$.
We define $C_*$ by the Hopf algebra which corepresents 
the group-valued functor $C(-)$.
In Section \ref{sec:boardman},
we survey stable comodule algebras and multiplicative operations
for  cohomologies
$E=H\Bbb{Z}/2,\ H\Bbb{Z}/p,\ KU,\ K(n),\ MU,\ BP$
from \cite{board}.
Let $\op{Op}_E(R_*)$ be the set of multiplicative operations of $E$ over
an $E_*$-algebra $R_*$.
We use Boardman's stable comodule algebras to study the functor $\op{Op}_E(-)$.
A {\em stable comodule algebra} $E^*(X)$
has the universal coaction $\rho_X:E^*(X)\to E^*(X)\HOT E_*(E)$.
It induces an isomorphism 
$\op{Op}_E (R_*)\cong \hom_{E_*\text{-alg.}}(E_*(E),R_*)$
of sets.
We are concerned with the case $E=H\Bbb{Z}/2,\ H\Bbb{Z}/p$ or $K(n)$.
Since  the stable cooperations $E_*(E) $ is a Hopf algebra in these cases, 
we  improve them slightly.
Then the multiplicative operations $\op{Op}_{E}(-)$ is
not only a set-valued functor, but also a group-valued functor.
In particular, there exists an isomorphism
$$\op{Op}_E (R_*)\cong \hom_{E_*\text{-alg.}}(E_*(E),R_*)$$
of groups.
Therefore, a similar approach to \cite{I,I2} can be applied for  $K(n)$.
In section \ref{sec:isom}, we show that $C_*$ is isomorphic to $K(n)_*(K(n))$
naturally.
Let $\CP^{\If}$  be the infinite complex projective space,
$\CP^{p^n-1}$ the $(p^n-1)$-dimensional complex projective space,
and $\BZN$ the $(2p^n-1)$-skeleton of the classifying space $B\Bbb{Z}/p$.
Substituting $X=\CP^{\If}$, $\CPN$, $\BZN$  into multiplicative operations
$\gamma_X$,
we  construct  natural homomorphisms
$$\kappa_1:\op{Op}(-)\lra \aut_{H_n}(-),\
\kappa_2:\op{Op}(-)\lra \aut_{\bar{g}_a}(-),\
\kappa_3:\op{Op}(-)\lra \aut_{\bar{G}_a}(-) $$
of groups.
By the definition of $C(-)$ and
$\alpha \circ \kappa_1=\beta \circ \kappa_3=\kappa_2$,
we obtain the natural transformation
$$\kappa: \op{Op}(-)\cong  \hom_{K(n)_*\text{-alg}}(K(n)_*(K(n)),-)
\lra C(-)\cong \hom_{K(n)_*\text{-alg}}(C_*,-) $$
of groups.
This yields the corresponding homomorphism $\kappa^*:C_*\to K(n)_*(K(n))$
of Hopf algebras.
It follows from the algebra structure of $K(n)_*(K(n))$
that $\kappa^*$ is an isomorphism.

I would like to thank  Daisuke Kishimoto and
 Atsushi Yamaguchi for useful  comments
on an earlier version of this work.
Further thanks are due to  Takeshi Torii for helpful discussions on this work
and  improvement of Definition \ref{defi:mulop}.
I am deeply grateful to Dr.~ Andrew Baker for  valuable advice on
previous researches.

\section{Formal group laws of height $n$}\label{sec:fgls}

In this section, we survey the Honda formal group law $H_n$ of height $n$
and automorphisms of $H_n$.
See \cite{R,W3} for more details.
We prepare some categories:
\begin{itemize}
\item $\mathbf{Gp}$: the category of groups;
\item $\mathbf{Alg}$: the category of graded commutative $K(n)_*$-algebras.
\end{itemize}
We express the coefficient ring of $n$th Morava $K$-theory $K(n)$
as $K(n)_*=\FP[v_n,v_n^{-1}]$ with $\deg(v_n)=-2(p^n-1)$.
Any algebra $R_*$ is an object of $\mathbf{Alg}$
except in Section \ref{sec:boardman}.
Note that every element of $R_n$ is of degree $-n$.
\begin{defi}\label{defi:Hn}
  Let $R_*$  be a  $K(n)_*$-algebra, and
  $x,y$ indeterminates of degree 2.
  \begin{enumerate}
\renewcommand{\theenumi}{\arabic{enumi}}
\renewcommand{\labelenumi}{(\theenumi)}
\item A homogeneous power series $H_n(x,y)\in  R_*[[x,y]]$ of degree 2 is called
  the {\em Honda formal group law}  of height $n$ over $R_*$
  if $H_n$ is the $p$-typical formal group law whose  $p$-series  is given by
     $[p]_{H_n}(x)=v_nx^{p^n}$.
\item 
  A homogeneous power series $f(x)\in  R_*[[x]]$
  of degree 2  is a {\em strict automorphism}
  of $H_n$
  if $f(x)$ is an automorphism of $H_n$ with
  $f(x)\equiv x \mod (x^2)$.
\item We denote by $\aut_{H_n}(R_*)$ 
the set of strict automorphisms  of $H_n$ over $R_*$.
\item \label{2.1-4}
  We define the multiplication on $\aut_{H_n}(R_*)$
by
\begin{equation*}
  \label{eq:prod-fgl}
f(x)\cdot g (x)=g\circ f(x).  
\end{equation*}
Then $\aut_{H_n}(-)$ is a functor from $\mathbf{Alg}$ to $\mathbf{Gp}$.
  \end{enumerate}
\end{defi}
It is well known  that the Hopf algebra $\Sigma(n)$ in \eqref{eq:kncoop}
corepresents the functor $\aut_{H_n}(-)$.
See \cite{BW2,BW,R}.
The multiplication of Definition \ref{defi:Hn} (\ref{2.1-4})
is different from the usual composition of automorphisms.
Therefore 
we use the conjugations $\bar{t}_i$ of $t_i$,
and we express $\Sigma(n)$ as
$$ \Sigma(n)\cong
K(n)_*[\bar{t}_1,\bar{t}_2,\dotsc]/(\bar{t}_i^{p^n}-v_n^{p^i-1}\bar{t}_i)
\quad (\deg \bar{t}_i=-(2p^i-2)). $$
The comultiplications of $\bar{t}_i$ correspond to the multiplication on $\aut_{H_n}(-)$
as follows.
\begin{thm}\label{lem:repHn}
  Let $\bar{t}_i$ be the elements as above, $\bar{t}_0=1$,
  and $R_*\in \mathbf{Alg}$.
  Assigning 
$$ f(x)=\sum_{i\geq 0} {}^{H_n}\theta(\bar{t}_i)x^{p^i} \in \aut_{H_n}(R_*)$$
to a homomorphism  $\theta\in \hom_{K(n)_*\text{-alg}}(\Sigma(n),R_*)$,
we obtain the natural isomorphism 
$$\hom_{K(n)_*\text{-alg}}(\Sigma(n),-)\cong \aut_{H_n}(-)
$$
of groups. Here $\sum {}^{H_n}$ is the formal sum of $H_n$.
  
\end{thm}

 The Hopf algebra   ${\Sigma} (n)$
 has relationship with $K(n)^*(\CP^{\infty})$ in Section \ref{sec:isom}.

\section{Additive formal group law chunks}\label{sec:afgls}

In \cite[Definition 5.7.1]{H}, 
a truncated form of a formal group law is called a
{\em formal group law chunk}.
We recall some definitions.
Since we treat  formal group laws over a graded algebra,
we modify them.

\begin{defi}
  Let  $R_*$ be a $\KN$-algebra. 
  \begin{enumerate}
\renewcommand{\theenumi}{\arabic{enumi}}
\renewcommand{\labelenumi}{(\theenumi)}    
  \item 
  Let $x$ and $y$ be indeterminates of degree 2.
  A graded {\em formal group law chunk} of size $n$ over  $R_*$ is a
  homogeneous element 
  $F(x,y)\in R_*[x,y]/(x,y)^{n+1} $ of degree 2
  which satisfies the conditions: 
  \begin{enumerate}
\renewcommand{\theenumii}{\roman{enumii}}
\renewcommand{\labelenumii}{(\theenumii)}    
\item $F(0,x)=F(x,0)=x $;
\item $F(x,y)=F(y,x) $;
\item $F(x,F(y,z))=F(F(x,y),z)$ in $R_*[x,y,z]/(x,y,z)^{n+1}$.
  \end{enumerate}
\item
  The  { formal group law chunk} $F(x,y)=x+y$ of size $n$ over  $R_*$ 
  is said to be {\em additive}.
  In particular, we denote the additive formal group law chunk of size $p^n-1$
  by
  $$\bar{g}_a(x,y)=x+y \in R_*[x,y]/(x,y)^{p^n}.$$
  \end{enumerate}
\end{defi}
\begin{defi}\label{defi:ga}
  \begin{enumerate}
\renewcommand{\theenumi}{\arabic{enumi}}
\renewcommand{\labelenumi}{(\theenumi)}    
  \item 
    Let $F,G $ be formal group law chunks of size $n$
    over a $\KN$-algebra $R_*$, and $x$ an indeterminate of degree 2.
  A {\em homomorphism} from $F$ to $G$ is a homogeneous element 
  $f(x)\in R_*[x]/(x)^{n+1} $ of degree 2
  which satisfies the conditions: 
    \begin{enumerate}
\renewcommand{\theenumii}{\roman{enumii}}
\renewcommand{\labelenumii}{(\theenumii)}    
  \item The constant term of $f(x) $ is 0;
  \item $f(F(x,y))=G(f(x),f(y)) $.
  \end{enumerate}
  If $f(x)\equiv x \mod (x)^2$, then it is called a {\em strict} isomorphism.
\item\label{3.4-2}
  We denote the strict automorphisms of $\bar{g}_a$ over $R_*$ by $ \aut_{\bar{g}_a}(R_*)$.
    The multiplication on $ \aut_{\bar{g}_a}(R_*)$ is defined as 
   Definition \ref{defi:Hn} (\ref{2.1-4}).
  Then  $\aut_{\bar{g}_a}(-) $ is a functor from $\mathbf{Alg}$ to $\mathbf{Gp}$.
\end{enumerate}
\end{defi}
\begin{prop}\label{lem:repga}
  The group-valued functor $\aut_{\bar{g}_a}(-)$ is corepresented by
  the Hopf algebra
  $$A_*=K(n)_*[\xi_1,\dotsc ,\xi_{n-1}]\quad \left( \deg \xi_i=-2(p^i-1)\right)$$
  whose comultiplication is given by
  $\xi_k\mapsto \sum_{i=0}^k \xi_{k-i}^{p^i}\otimes \xi_i$.
  Here $\xi_0=1$.
  In other words,
  we have the natural isomorphism
  $$\hom_{\KN\text{-alg}}(A_*,R_*)\cong \aut_{\bar{g}_a}(R_*)  \quad
  \left(\theta\mapsto f(x)=x+\sum_{i= 1}^{n-1}\theta (\xi_i)x^{p^i}\right) $$
  of groups.
\end{prop}
\begin{proof}
  An automorphism $f(x)$ of $\bar{g}_a$  satisfies $f(x)+f(y)=f(x+y)$.
  Thereby,  a strict automorphism $f(x)$ of $\bar{g}_a$ is expressed as
\begin{equation}
  \label{eq:autfglc}
f(x)=\sum_{k=0}^{{n-1}}a_{k} x^{p^k}\in R_*[x]/(x^{p^n}),  
\end{equation}
where $a_0=1$ and $a_k\in R_{2p^k -2}$ for $0<k\leq n-1$.
The multiplication $f(x)\cdot g(x)=g\circ f(x)$
induces the comultiplication on $A_*$.
\end{proof}
 The Hopf algebra   $A_*$
 is connected with $K(n)^*(\CPN)$ in Section \ref{sec:isom}.

\section{2-dimensional additive formal group law chunks}\label{sec:2fgl}
In  \cite[Chapter II]{H}, Hazewinkel  defined
higher dimensional formal group laws and homomorphisms
over a commutative ring.
In \cite[Appendix]{I2}, the author  introduced
a graded version of these definitions 
to describe the  dual Steenrod algebra $H_*(H)$ for an odd prime $p$.
In particular, we   need 
the graded 2-dimensional additive formal group law $G_a$
over a graded $\FP$-algebra, which is given by \eqref{eq:2dimAFGL}.
The group of  quasi-strict automorphisms  of $G_a$  induces the
Hopf algebra  $H_*(H)$ including the exterior part $E(\tau_0,\tau_1,\dotsc)$.

We  employ a
graded 2-dimensional additive formal group law chunk $\bar{G}_a$,
the truncated version of $G_a$,
which yields the exterior part of $K(n)_*(K(n))$.
We  give only the definition of graded  2-dimensional formal group law chunks
which are needed for our purpose.
In particular, we need the additive one.
See \cite[Appendix]{I2} for more details on  
graded higher dimensional formal group laws.
\begin{defi}\label{defi:Ga2}
  Given a $\KN$-algebra $R_*$, we fix the $R_*$-algebra
  $ E(\ee_1,\ee_2)\otimes R_*[x_1,x_2]/(x_1,x_2)^{p^n} $
  with $\deg \ee_i=1,\deg x_i=2$ for $i=1,2$.
  \renewcommand{\theenumi}{\arabic{enumi}}
\renewcommand{\labelenumi}{(\theenumi)}
  \begin{enumerate}
  \item\label{defi:Ga2-1}    A pair 
   $$\bar{F}(X_1;X_2)=\bar{F}(\ee_1,x_1;\ee_2, x_2)
   =(\bar{F}(1)(\ee_1,x_1;\ee_2, x_2),\bar{F}(2)(\ee_1,x_1;\ee_2, x_2)) $$
   of elements of $E(\ee_1,\ee_2)\otimes  R_*[x_1,x_2]/(x_1,x_2)^{p^n}$ 
   is called a graded {\em  2-dimensional  formal group law chunk} over $R_*$
   if it satisfies the following conditions:
   \begin{enumerate}
  \renewcommand{\theenumii}{\roman{enumii}}
\renewcommand{\labelenumii}{(\theenumii)}
   \item $\bar{F}(i)$ is a homogeneous element of degree $i$ for $i=1,2$;
   \item $\bar{F}(1)\equiv \ee_1+\ee_2,\ \bar{F}(2)\equiv x_1+x_2\mod (x_1,x_2)^2$;
   \item 
     $\bar{F}(i)(X_1;X_2)=\bar{F}(i)(X_2;X_1)$;
   \item $F(i)(F(i)(X_1;X_2);X_3)=F(i)(X_1;F(i)(X_2;X_3))$
     in $E(\ee_1,\ee_2,\ee_3) \otimes R_*[x_1,x_2,x_3]/(x_1,x_2,x_3)^{p^n}$.
   \end{enumerate}
 \item \label{defi:Ga2-2} The graded   2-dimensional  formal group law chunk
   $$\bar{G}_a(X_1;X_2) 
   =(\bar{G}_a(1)(\ee_1,x_1;\ee_2, x_2),\bar{G}_a(2)(\ee_1,x_1;\ee_2, x_2))=
   (\ee_1+\ee_2,x_1+x_2)$$
is said to be {\em additive}. 
 \end{enumerate}
\end{defi}
\begin{defi}\label{defi:Ga2hom}
Let  $R_*[x]/(x^{p^n})\otimes E(\ee)$ be the graded algebra
with $\deg\ee=1,\ \deg x=2$.
  \begin{enumerate}
  \renewcommand{\theenumi}{\arabic{enumi}}
\renewcommand{\labelenumi}{(\theenumi)}
\item \label{4.2-0}
  Let $\bar{F}(X_1;X_2),\bar{F}'(X_1;X_2)$ be
  graded 2-dimensional formal group law chunks over $R_*$.
  A pair $$f(X)=f(\ee,x)=(f(1)(\ee,x),f(2)(\ee,x)) $$ of elements of
    $E(\ee)\otimes R_*[x]/(x^{p^n})$ is
     a  {\em homomorphism} from
     $\bar{F}$ to $\bar{F}'$ if it satisfies the following conditions:
     \begin{enumerate}
  \renewcommand{\theenumii}{\roman{enumii}}
  \renewcommand{\labelenumii}{(\theenumii)}
     \item $f(i)$ is a homogeneous element of degree $i$ for $i=1,2$;
     \item The constant terms of $f(1)$ and $f(2) $ are 0;
     \item $f(\bar{F}(X_1;X_2))=\bar{F}'(f(X_1);f(X_2))$.
     \end{enumerate}
\item \label{4.2-1}
  An endmorphism  $f(X)=(f(1)(\ee,x),f(2)(\ee,x)) $ of
   $\bar{G}_a$ in Definition \ref{defi:Ga2}(\ref{defi:Ga2-2}) is 
     a  {\em quasi-strict} automorphism over $R_*$
     if it satisfies 
     \begin{equation}
       \label{eq:4.2-1-2}
     f(1)(\ee,x)\equiv \ee +a_0x,\ f(2)(\ee,x)\equiv x \mod (x)^2,        
     \end{equation}
  where $a_0 \in R_1 $.
\item \label{4.2-3:prod-2afgl}
  We denote the quasi-strict automorphisms of $\bar{G}_a$ over $R_*$ by $  \aut_{\bar{G}_a}(R_*)$. The multiplication on $ \aut_{\bar{G}_a}(R_*)$ is defined as 
\begin{equation*}
  \label{eq:prod-2afgl}
f(X)\cdot g(X)=g\circ f(X)=(g(1)(f(1)(X),f(2)(X)),
g(2)(f(1)(X),f(2)(X))), 
\end{equation*}
where $f(X)=(f(1)(X),f(2)(X)),\ g(X)=(g(1)(X),g(2)(X))\in \aut_{\bar{G}_a}(R_*)$.
\end{enumerate}
\end{defi}
If an endmorphism $f(X)$ of $\bar{G}_a$ satisfies \eqref{eq:4.2-1-2},
then it is an invertible element of $\aut_{\bar{G_a}}(R_*) $ and
therefore an automorphism of $\bar{G}_a$.
The term "quasi-strict" is used in   \cite{I2}.
Thus, we obtain the functor
$$\aut_{\bar{G}_a}(-):\mathbf{Alg}\lra \mathbf{Gp}. $$
An automorphism $f(X)=(f(1)(\ee,x),f(2)(\ee,x)))$ of $\bar{G}_a$
satisfies 
\begin{equation}\label{eq:autGa}
  \begin{aligned}
  f(1)(\ee_1+\ee_2,x_1+x_2)&=f(1)(\ee_1,x_1)+f(1)(\ee_2,x_2)
  \\
   f(2)(\ee_1+\ee_2,x_1+x_2)&=f(2)(\ee_1,x_1)+f(2)(\ee_2,x_2).
  \end{aligned}
\end{equation}
The functor $\aut_{\bar{G}_a}(-)$ is corepresented as follows.
\begin{thm}\label{prop:repGa}
Let
$$B_*= K(n)_*[\xi_1,\dotsc ,\xi_{n-1}]\otimes E(\bar{\tau}_0,\dotsc ,\bar{\tau}_{n-1})
\quad (\deg \bar{\tau}_n=-(2p^n-1), \deg \xi_n=-2(p^n-1))$$ 
be the Hopf algebra whose comultiplication is given by
\begin{align}\label{eq:coprodB}
\bar{\tau}_k\mapsto  \bar{\tau}_k\otimes 1+\sum_{i=0}^k \xi_{k-i}^{p^i}\otimes \bar{\tau}_i,\
\xi_k\mapsto \sum_{i=0}^k \xi_{k-i}^{p^i}\otimes \xi_i,
\end{align}
where $\xi_0=1$.
Then
$B_*$ 
corepresents the group-valued functor $\aut_{\bar{G}_a}(-)$.
In other words, we have the natural  isomorphism of groups
$$\hom_{K(n)^*\text{-alg}}(B_*,-)\cong \aut_{\bar{G}_a}(-) 
\quad \left(\theta\mapsto \left(f(1),f(2)\right)
=\left(\ee+\sum_{i=0}^{n-1}\theta(\bar{\tau}_i)x^{p^{i}},
x+\sum_{i=1}^{n-1}\theta(\xi_i)x^{p^{i}}
\right)\right).$$
\end{thm}
\begin{proof}
  A quasi-strict automorphism $f(X)=(f(1)(\ee,x),f(2)(\ee,x))$
  satisfies the equalities
  \eqref{eq:autGa}.
  By a similar calculation to \cite[Appendix]{I2},
it is easily seen that 
$f(X) $
is expressed as
\begin{equation}\label{eq:2autfglc}
\begin{aligned}
  f(1)(\ee,x)&=\ee+a_0x +a_1x^p+\cdots  +a_{n-1}x^{p^{n-1}}
               =\ee+\sum_{i=0}^{n-1}a_ix^{p^i}\in
  E(\ee)\otimes R_*[x]/(x)^{p^n}\\
  f(2)(\ee,x)&=\phantom{\ee+a_0}x +b_1x^p+\cdots  +b_{n-1}x^{p^{n-1}}
=\phantom{\ee + \ } \sum_{i=0}^{n-1}b_ix^{p^i}              \in E(\ee) \otimes R_*[x]/(x)^{p^n},
\end{aligned}
\end{equation}
where $a_i\in R_{2p^i-1},\ b_i\in R_{2p^i-2}$ and $b_0=1  $.
Given a quasi-strict automorphism $g(X)$ with
$$g(1)(\ee,x)=\ee+\sum_{j=0}^{n-1}a'_jx^{p^j},\ 
g(2)(\ee,x)=\sum_{j=0}^{n-1}b'_jx^{p^j}\ (b'_0=1),
$$
the multiplication $f(X)\cdot g (X)$ consists of
\begin{align*}
  (f\cdot g)(1)(\ee, x)&=g(1)(f(1)(\ee ,x),f(2)(\ee, x))=
  \left(\ee +\sum_{i=0}^{n-1}a_ix^{p^i} \right)
    +\sum_{j=0}^{n-1}a'_j\left( \sum_{i=0}^{n-1}b_ix^{p^i}\right)^{p^j} \\
  &=\ee +\sum_{k=0}^{n-1}\left( a_k+\sum_{j=0}^kb_{k-j}^{p^j}a_j'\right)x^{p^k},\\
    (f\cdot g)(2)(\ee, x)&=g(2)(f(1)(\ee ,x),f(2)(\ee, x))=
    \sum_{j=0}^{n-1}b'_j\left( \sum_{i=0}^{n-1}b_ix^{p^i}\right)^{p^j}
    =\sum_{k=0}^{n-1}\left(\sum_{j=0}^{k}b_{k-j}^{p^j}b'_j \right)x^{p^k}.
\end{align*}
These equalities  give the comultiplication \eqref{eq:coprodB}.
\end{proof}

The Hopf algebra $B_*$ is related to  $K(n)_*(\BZN)$
in Section \ref{sec:isom}.

\section{The fiber product of automorphism groups of  formal group laws}
\label{sec:fiber}
In this section, we study relations among the functors
$\aut_{H_n}(-),\ \aut_{\bar{g}_a}(-)$ and $\aut_{\bar{G}_a}(-)$
in Sections \ref{sec:fgls}, \ref{sec:afgls} and \ref{sec:2fgl}, respectively.
The relations lead to  a  Hopf algebra $C_*$
which is isomorphic to $K(n)_*(K(n))$.
We construct  the Hopf algebra $C_*$ algebraically
since we use only formal group laws.

First, we examine relations between $H_n$ and $\bar{g}_a$.
The $p$-series of the Honda formal group law $H_n(x,y)$
over $\KN$ is $[p]_{H_n}(x)=v_nx^{p^n}$.
Via the quotient map
$$K(n)_*[[x,y]]\lra K(n)_*[x,y]/(x,y)^{p^n},$$
we regard $H_n$ as a formal group law chunk $\bar{H}_n$ of size $p^n-1$.
By comparison of degrees, we have
$\bar{H}_n(x,y)=\bar{g}_a(x,y)=x+y \in K(n)_*[x,y]/(x,y)^{p^n} $.
Given a $K(n)_*$-algebra $R_*$, a strict automorphism
$f(x)\in \aut_{H_n}(R_*)$ maps to
the strict automorphism $\bar{f}(x)\in \aut_{\bar{g}_a}(R_*)$
through the projection $R_*[[x]]\to R_*[x]/(x^{p^n})$.
Explicitly, it is expressed as
\begin{equation}\label{eq:Hntoga}
f(x)=\sum_{i=0}^{\If}{}^{H_n}a_i x^{p^i}\mapsto
\bar{f}(x)=\sum_{i=0}^{n-1}a_i x^{p^i},
\end{equation}
where $a_0=1$.
\begin{prop}
  \label{prop:Hntoga}

Let $R_*$ be a $K(n)_*$-algebra.  
The map \eqref{eq:Hntoga} gives
the natural transformation 
$$\alpha:\aut_{H_n}(R_*)\cong \hom_{K(n)_*\text{-alg.}}(\Sigma(n),R_*)
\lra \aut_{\bar{g}_a}(R_*)\cong \hom_{K(n)_*\text{-alg.}}(A_*,R_*)$$
of groups. Moreover $\alpha$
induces the Hopf algebra homomorphism
$$\alpha^*:A_*\cong K(n)_*[\xi_1,\dotsc,\xi_{n-1}]\lra
{\Sigma}(n)\cong K(n)_*[\bar{t}_1,\bar{t}_2,\dotsc]/
(\bar{t}_i^{p^n}-v_n^{p^i-1}\bar{t}_i) $$
which is defined by $\alpha^*(\xi_i)=\bar{t}_i$.
  
\end{prop}
\begin{proof}
  The latter part follows from
  Theorem  \ref{lem:repHn}, Proposition \ref{lem:repga} and Yoneda's Lemma.
\end{proof}

Next, we discuss relations between $\bar{G}_a$ and $\bar{g}_a$.
\begin{prop}\label{prop:Gatoga}
  Let $R_*$ be a $K(n)_*$-algebra.  
We define the natural  homomorphism  of groups 
$$\beta:\aut_{\bar{G}_a}(R_*)\cong \hom_{K(n)_*\text{-alg.}}(B_*,R_*)\lra 
\aut_{\bar{g}_a}(R_*)\cong \hom_{K(n)_*\text{-alg.}}(A_*,R_*)$$
by
$$f(\epsilon,x)=(f(1)(\epsilon,x),f(2)(\epsilon,x))\mapsto f(2)(\ee,x). $$
This corresponds to
the  homomorphism of Hopf algebras
$$ \beta^*:A_*=K(n)_*[\xi_1,\dotsc,\xi_{n-1}] \lra
B_*=K(n)_*[\xi_1,\dotsc ,\xi_{n-1}]\otimes E(\bar{\tau}_0,\dotsc ,\bar{\tau}_{n-1})$$
which is given by $\xi_i\mapsto \xi_i$.
\end{prop}

\begin{proof}
An element $f(\epsilon,x)=(f(1)(\epsilon,x),f(2)(\epsilon,x))$
of $\aut_{\bar{G}_a}(R_*)$ is expressed as  the equalities  \eqref{eq:2autfglc}.
It follows from the equality \eqref{eq:autfglc} that
$f(2)(\ee,x)\in R_*[[x]]$ is a strict automorphism
of $\bar{g}_a$.
By Proposition \ref{lem:repga} and Theorem \ref{prop:repGa},
we see that $\beta$ induces the homomorphism $\beta^*$.
\end{proof}

\begin{defi}\label{defi:fp}
  We define  the functor $C(-)$ by the fiber product of $\alpha$ and $\beta$.
  In other words,
  the group $C(R_*)$ is  the fiber product 
\begin{equation}\label{eq:fp}
\xymatrix{
  C(R_*)\ar@{.>}[r]\ar@{.>}[d] & \aut_{H_n}(R_*)\ar[d]^{\alpha}\\
  \aut_{\bar{G}_a}(R_*)\ar[r]^{\beta}& \aut_{\bar{g}_a}(R_*)
}
\end{equation}
of groups.
\end{defi}
The functor $C(-)$ is corepresentable as follows.
\begin{thm}\label{thm:fp}
  The group-valued functor $C(-)$ is corepresented by
  the Hopf algebra
  $$C_*= {\Sigma}(n)\otimes_{A_*}B_*\cong  {\Sigma}(n)\otimes_{K(n)_*}E(\bar{\tau}_0,\dotsc ,\bar{\tau}_{n-1}).$$
  That is, there exists a natural isomorphism
  $C(-)\cong \hom_{K(n)_*\text{-alg.}}(C_*,-)$.
  The comultiplication on $\bar{\tau}_i$
  is given by
  \begin{align*}
  \Delta(\bar{\tau}_i)=\bar{\tau}_i\OT 1 +\sum_{j=0}^i t_{i-j}^{p^j}\OT \bar{\tau}_j.
\end{align*}
\end{thm}
\begin{proof}
  Since $\alpha$ and $\beta$ are natural transformations
  of  groups, 
the fiber product $C(R_*)$ 
inherits the group structure from $\alpha$ and $\beta$.
The fiber product \eqref{eq:fp} implies
that $C(-)$ is  corepresented by the pushout $C_*$ of
the diagram
$$\xymatrix{
  C_* & {\Sigma}(n)_*\ar@{.>}[l]\\
  B_*\ar@{.>}[u]& A_*\ar[l]^{\beta^*}\ar[u]^{\alpha^*}  
}$$
of Hopf algebras.
In other words, there exists
a Hopf algebra $C_*$ such that
$$C(-)\cong \hom_{K(n)_*\text{-alg}}(C_*,-),$$
and $C_*$ is isomorphic to
$${\Sigma}(n)\otimes_{A_*}B_*\cong
{\Sigma}(n)\otimes_{K(n)_*}E(\bar{\tau}_0,\dotsc ,\bar{\tau}_{n-1}). $$
The comultiplication $\Delta(\bar{\tau}_i)$ derives
from the comultiplication on $B_*$ in Theorem \ref{prop:repGa}.
\end{proof}
This suggests that the Hopf algebra $K(n)_*(K(n))$ is expressed in terms
of formal group laws.
Note that $\bar{\tau}_i \in B_*$ corresponds to 
the conjugation of $\tau_i\in \KN (K(n))$ in \eqref{eq:kncoop}.

\section{Boardman's stable comodule algebras}
\label{sec:boardman}

This section is devoted to a summary of \cite{board}.
We recall  stable comodules, stable comodule algebras and multiplicative operations from  \cite{board},
and improve some of them.
In \cite{board}, 
Boardman  mainly dealt with stable cooperations
of the cohomologies
$$E=H\Bbb{Z}/2,\ H\Bbb{Z}/p,\ K(n),\ KU, \ MU,\ BP,$$
where $MU$ is the complex cobordism, and
$KU$ is the complex $K$-theory.
In these cases, the stable cooperations $E_*(E)$ is a free (left or right)  $E^*$-module and a Hopf algebroid.
If $E=H\Bbb{Z}/2,\ H\Bbb{Z}/p$ or $ K(n)$,
then $E_*(E)$ becomes a Hopf algebra.
In this section, we treat  the six cohomologies above
to clarify  differences between them,
and study $E_*(E)$ which is a Hopf algebra.

Throughout this section,
we denote  tensor products and  completed tensor products over $E_*$
by  $\OT$ and $\hat{\OT}$, respectively. 
We may write $(X,o)$
for  a space with  basepoint or a spectrum.
We introduce some categories.
We use the same notation as \cite{board}
except $\mathbf{GAb}$, $\mathbf{GAb}^*$ and $\mathbf{Alg}_E$.
See \cite[Section 6]{board} for more details.
\begin{itemize}
\item $\mathbf{Ho}$: The homotopy category of spaces
  which are homotopy equivalent to CW complexes.
  The morphisms $\mathbf{Ho}(X,Y)$ is the homotopy set $[X,Y]$.
\item $\mathbf{Stab} $:  The stable homotopy category of spectra
with maps of degree $0$ as morphisms.
The  morphisms $\mathbf{Stab}(X,Y)$ is defined by the group
of  homotopy classes $\{ X,Y \}$.
\item $\mathbf{Stab^*}$: The {\em graded} stable homotopy category.
It has the same objects as $\mathbf{Stab}$ with maps of any degree as morphisms.
We express the group of maps of degree $n$ as
$$\mathbf{Stab}^n(X,Y)=\{ X,Y\}^n=\{ X,\Sigma^n Y \}=\{ \Sigma^{-n} X,Y \},$$
where $\Sigma $ is the reduced suspension.
Let  $(S^k,o)$ be the $k$-sphere spectrum.
Fixing an isomorphism $S^1\simeq S^0$  of degree $1$ in $\mathbf{Stab}^*$,
we define the canonical natural {\em desuspension isomorphism}
\begin{equation}\label{eq:sus}
\Sigma X  =S^1\wedge X\simeq S^0\wedge X\simeq X
\end{equation}
of degree $1$ for any spectrum $X$.
\item $\mathbf{GAb}^*$: The category of {\em graded} abelian groups and
  homomorphisms of any degree.
  A homomorphism $f:A\to B$ of degree $n$ has the components
  $f^i:A^i\to B^{i+n}$.
  Given a graded abelian group $A$ and an integer $n$,
  we define the {\em $n$-fold suspension} $\Sigma^n A$ of $A$
  by shifting every element up in degree by $n$.
  That is, $(\Sigma^n A)^i$ consists of the elements $\Sigma^n x$
  for $x\in A^{i-n}$.
  Thus $\Sigma^n:A\to \Sigma^n A$
  is an isomorphism of degree $n$.
\item $\mathbf{GAb}$: The category with the same objects as $\mathbf{GAb}^*$
and homomorphisms of degree 0.  
\item $\mathbf{Mod}^*$: The category of $E^*$-modules with maps of any degree as morphisms.
  An $E^*$-module homomorphism $f:M\to N$ of degree $n$
  has components $f^i:M^i\to N^{i+n}$ which satisfy $f^{i+h}(vx)=(-1)^{nh}vf^i(x)$
  for $x\in M^i$ and $v\in E^h$.
  The group $\mathbf{Mod}^n(M,N)$ consists of all $E^*$-homomorphisms of
  degree $n$.
  The morphisms $\mathbf{Mod}^*(M,N)$ is an $E^*$-module in an obvious way.
  The suspension isomorphism $\Sigma^n: M\to \Sigma^n M$ of $E^*$-modules
  of degree $n$ satisfies $v(\Sigma^n x)=(-1)^{nh}\Sigma^n (vx)$
  for $v\in E^h$  and $x\in M$.

\item $\mathbf{FMod}^* $: The category of
complete Hausdorff {\em filtered} $E^* $-modules  and
continuous $E^* $-module homomorphisms of any degree.
An object is a complete Hausdorff $E^* $-module  with the filtration.
That is, an object $M$ has $E_*$-submodules $F^aM\subset M $,
and satisfies $M\cong \invlim_a M/F^aM $ to make the topology complete Hausdorff.
\item $\mathbf{FMod} $:  The category with
all objects of $\mathbf{FMod}^*$ and continuous $E^* $-module homomorphisms of degree 0.
\item $\mathbf{FAlg}$: The category of  complete Hausdorff commutative
   filtered $E^*$-algebras.
A {\em filtered $E^*$-algebra} $A$ has
the multiplication $\phi:A\OT A\to A$, the unit $E^*\to A$
and the {\em filtration of  ideals} $F^aA$ of $ A$.
Then $\phi$ is continuous, and
we identify $\phi$ with the completed multiplication $A\hat{\OT} A \to A$.
An object $A$ satisfies $A\cong \invlim_a A/F^aA $.
\item  $\mathbf{Alg}_E$: The full subcategory  of $\mathbf{FAlg}$
  which consists of  objects $M$ with trivial filtration $\{ 0\}\subset M$.
Any object of $\mathbf{Alg}_E$ is a discrete and complete $E^*$-module.
\end{itemize}
For a filtered $E^* $-module $M$ with $F^aM\subset M $,
we define the completion  $\WH{M}=\invlim_aM/F^aM $ of $M$.
Then $\WH{M}$ is a complete filtered $E^* $-module,
and we have the completion functor
$$(-)\hat{\ }:\mathbf{Mod}\to \mathbf{FMod}.$$
Let $M,N$ be complete Hausdorff filtered $E^* $-modules.
Since the usual tensor product $M\otimes N$
is rarely complete,
we define the completed tensor product $M\hat{\otimes} N $ in $\mathbf{FMod}$
as
$$ M\hat{\otimes}N= \invlim_{a,b}[(F/F^aM)\otimes (N/F^bN)],$$
where $F^aM,\ F^b N$ are the filtrations of $M,\ N$.
Treating any $E_*$-algebra as a  filtered algebra with trivial filtration,
we identify $\mathbf{Alg}_E$ with the category of $E^*$-algebras.
If $E=K(n)$, then $\mathbf{Alg}_{K(n)}$ coincides with  $\mathbf{Alg}$
defined in Section \ref{sec:fgls}.

The cohomology $E^*(X)$ of a CW-complex $X$ has the {\em profinite  filtration}
$$ F^aE^*(X)=\ker[E^*(X)\lra E^*(X_\alpha)],$$
where $X_a$ runs through all finite subcomplexes of  $X$. 
We call the resulting filtration topology  the {\em profinite topology}.
The completion $E^*(X)\hat{\ }$
is an object of $\mathbf{FAlg}$, and
we obtain the  cohomology functor
$$ E^*(-)\hat{\ }:\mathbf{Ho}^{op}\lra \mathbf{FAlg}.$$
Here $\mathbf{Ho}^{op}$ is the opposite category of $\mathbf{Ho}$. 
The cohomology $E^*(X,o)$ of a spectrum $(X,o)$ is an $E_*$-module
with the profinite filtration by using all finite subspectra of $X$ similarly.
The completed cohomology $E^*(X,o)\hat{\ }$ is an object of $\mathbf{FMod}$.
The homology $E_*(X)$ is a discrete $E^*$-module with trivial filtration.

Given a discrete $E^*$-module $M$, the {\em dual-finite filtration} on $DM=\hom^*(M,E^*)$ in \cite[Definition 4.8]{board} consists of the submodules
$F^LDM=\op{Ker}[DM\to DL]$, where $L $ runs through all finitely generated submodules of $M$.
It gives the {\em dual-finite topology} on $DM$
and it is complete Hausdorff.
For any finite generated submodule $L\subset E_*(X)$,
there exists a finite subspectrum $X_{\alpha}$ of $X$
such that $L\subset E_*(X_{\alpha})$.
Hence, the duality morphism $d:E^*(X)\to DE_*(X)$, which is constructed
from the Kronecker pairing $\langle -,- \rangle:E^*(X)\OT E_*(X)\to E^*$,
is continuous.
\begin{thm}[\cite{board}Theorem 4.14]
  Assume that $E_*(X)$ is a free $E^*$-module.
  Then $X$ has {\em strong duality}, that is, $d:E^*(X)\cong DE_*(X)$
  is a homeomorphism between the profinite topology on $E^*(X)$
  and the dual-finite topology on $DE_*(X)$.
  In particular, $E^*(X)$ is complete Hausdorff.
\end{thm}
We deal with the spectrum $E$ whose cooperations  $E_*(E)$ is free
as an $E_*$-bimodule.
Therefore we have the homeomorphism $E^*(E)\cong DE_*(E)$.

The fixed desuspension isomorphism \eqref{eq:sus}
induces the $E_*$-isomorphism of  degree $1$
\begin{equation}\label{eq:susE}
  \Sigma:E^*(X,o)\to E^{*+1}(\Sigma X,o).
\end{equation}
Hence the cohomology $E^*(-)$ is a functor from $\mathbf{Stab}^*$ to $\mathbf{Mod}^*$, and we have the completed cohomology functor
$$ E^*(-)\hat{\ }:(\mathbf{Stab}^*)^{op}\lra \mathbf{FMod}^*.$$
For the one point space $T$ and the unit element $1_T\in E^0(S^0,o)\cong E^0(T)\cong E^0$,
we define the canonical generator $u_1=\Sigma 1_T\in E^1(S^1,o)$.
Similarly, we set  the canonical generator
\begin{equation}\label{eq:canogen}
u_k=\Sigma^k 1_T\in E^k(S^k,o)\subset E^k(S^k).  
\end{equation}
The  ring spectrum $E$ induces  the multiplication 
$$\phi:E_*(E)\OT E_*(E)\overset{\wedge}{\lra} E_*(E\wedge E) {\lra} E_*(E),$$
where $\wedge$ is the smash product.
Thus  $E_*(E)$ is a commutative $E^*$-algebra with the left or right $E^*$-action.
Because of ambiguity of tensor products and bimodules,
we use the  notations of \cite{board} as follows.
We rewrite the multiplication $\phi$ as
$$\phi:E1_*(E2)\OT_1 E1_*(E3)\overset{\wedge}{\lra} E1_*(E2\wedge E3)\lra
E1_*(E), $$
which we call the {\em $E^*$-action scheme} of $\phi$.
Here $Ei$ is a copy of $E$ tagged for identification,
and $\OT_i$ means  $\OT_{Ei_*}$.
The algebra $E^*(E)$ of stable operations acts on the cohomology $E^*(X)$ or
$E^*(X,o)$ by composition.
The action
\begin{equation}
  \label{eq:univaction}
\lambda_X:E^*(E)\OT E^*(X)\lra E^*(X),   
\end{equation}
 is represented as
$\lambda_X:E1^*(E2)\OT_2 E2^*(X)\lra E1^*(X). $
In \cite[(10.7)]{board}, Boardman   introduced
the partial dualization \eqref{eq:BoaMil} of $\lambda_X$ as follows.
The action $\lambda_X$ induces the $E2_*$-module homomorphism
$$\rho_X':E2^*(X)\lra \mathbf{FMod}^*(E1^*(E2),E1^*(X)\hat{\ }) $$
which is defined by $\rho_X'(x)=x^*:E^*(E)\to E^*(X)\hat{\ }$ for $x\in E^*(X)$.
In terms of an $E^*$-basis $\{c_\alpha|\alpha\in \Lambda \}$ of $E_*(E)$,
we define $u\in E^*(E)\HOT E_*(E) $
by $u=\sum_\alpha (-1)^{\deg c_\alpha}c_\alpha^*\OT c_\alpha$.
Here $c_\alpha^*\in E^*(E)$ is the linear functional dual to $c_\alpha$, which is given by
the Kronecker pairing $\langle c_\alpha^*,c_\alpha \rangle=1$ and $\langle c_\alpha^*,c_\beta \rangle$
for $\alpha\ne \beta.$
By \cite[Lemma 6.16(b)]{board} and the strong duality $E^*(E)\cong DE_*(E)$,
we have
$$\mathbf{FMod}^*(E^*(E),E^*(X)\hat{\ })\cong E^*(X)\HOT E_*(E)\quad
(g \mapsto (g^* \OT 1)(u)).$$
We obtain the partial dualization of $\lambda_X$
\begin{equation}
  \label{eq:BoaMil}
\rho_X: E^*(X)\lra E^*(X)\hat{\otimes} E_*(E)  
\end{equation}
whose action scheme is
$$\rho_X:E2^*(X)\to E1^*(X)\HOT_1 E1_*(E2).$$
Note that $E_*(E)$ is a discrete $E^*$-module.
The coaction $\rho_X$ generalizes the Milnor coaction 
$\lambda^*:H^*(X)\to H^*(X)\HOT H_*(H)$ in \cite[Section 4]{mil}.
We can recover the action of $r\in E^*(E)$ on $E^*(X)$ from $\rho_X$
as
$$r:E^*(X)\xrightarrow{\rho_X}E^*(X)\HOT E_*(E)\xrightarrow{1\OT \langle r,- \rangle} E^*(X)\OT E_*\cong E^*(X).$$ 
Since $\rho_X$ preserves the right action of $E2_*$,
the completion
\begin{equation}
  \label{eq:univcoa-1}
\rho_X: E2^*(X)\hat{}\lra E1^*(X)\hat{\otimes}_1 E1_*(E2)  
\end{equation}
is a natural morphism in $\mathbf{FMod}$.
Similarly, we  define the Milnor coaction for a spectrum $(X,o)$,
which is denoted by
\begin{equation}
  \label{eq:univcoa-2}
  \rho_{(X,o)}:E2^*(X,o)\hat{}\lra E1^*(X,o)\HOT_1 E1_*(E2).
\end{equation}
In this section, some coactions are defined on both spaces and spectra.
When we define coactions, we may explain them only on either spaces or spectra.

The composition and the unit of $E^*(E)$ have the action schemes
\begin{equation*}
  E1^*(E2)\OT_2 E2^*(E3)\overset{\circ}{\lra} E1^*(E3),\quad
  \xi:E1^*\lra E1^*(E2).
\end{equation*}
By dualization,
we obtain the comultiplication and  the counit
\begin{equation}
  \label{eq:Ecoa}
  \psi:E1_*(E3)\to E1_*(E2)\OT_2 E2_*(E3),
  \quad
  \ee:E_*(E)=\pi_*(E\wedge E)\lra E_*=\pi_*(E).
\end{equation}
Since the composition $\circ$ 
is associative,
the comultiplication $\psi$ 
is coassociative.
Therefore $E_*(E)$ is a Hopf algebroid.
Interchanging the first  and the second spectrum,
we obtain the conjugation
$$\chi:E1_*(E2)\to E2_*(E1).$$

Since the cohomology $E^*(X)$ or $E^*(X,o)$
is an $E^*(E)$-module with the action $\lambda$, 
it is an $E_*(E)$-comodule
with the coaction \eqref{eq:univcoa-1} or \eqref{eq:univcoa-2}.
By \cite[Theorem 11.14]{board},
the complete cohomology $E^*(X)\hat{\ }$ with the coaction $\rho_X$ is a stable comodule in the following sense.

\begin{defi}[\cite{board}Definition 11.11]\label{defi:stacom}
  A {\em stable ($E^*$-cohomology) comodule } 
  is a complete Hausdorff filtered $E^*$-module $M$
   equipped with a  morphism 
  $$\rho_M:M\lra M\hat{\otimes}E_*(E)\qquad
  (M2\lra M1\hat{\OT}_1 E1_*(E2))$$
   that is $E^*$-linear and continuous, and
  satisfies the axioms:
  \begin{equation}\label{eq:comstacom}
  \xymatrix{
    M3\ar[r]^-{\rho_M}\ar[d]^{\rho_M}&M1\HOT_1 E1_*(E3)\ar[d]^{M\otimes \psi}\\
    M2\HOT_2 E2_*(E3)\ar[r]_-{\rho_M\OT 1}&M1\HOT_1 E1_*(E2)\HOT_2 E2_*(E3),
  }\
 \xymatrix{
    M\ar[r]^-{\rho_M}\ar[rd]_{\cong}&M\HOT E_*(E)\ar[d]^{M\otimes \epsilon}\\
    &M\HOT E_*\cong M.
 }
  \end{equation}
\end{defi}
Note that $M$ is an object of $\mathbf{FMod}$ and $\rho_M$ is a morphism in
$\mathbf{FMod}$.

\begin{thm}[\cite{board}Theorem 11.14]\label{thm:univcom}
  \begin{enumerate}
\renewcommand{\theenumi}{\arabic{enumi}}
\renewcommand{\labelenumi}{(\theenumi)}
  \item \label{6.2-1}For any space $X $ or spectrum $(X,o)$,
    a natural coaction $\rho$ in \eqref{eq:univcoa-1} or \eqref{eq:univcoa-2}
    is a morphism in $\mathbf{FMod}$.
    It makes $E^*(X)\hat{} $ or $E^*(X,o)\hat{}$ a stable comodule.
  \item\label{6.2-2}
    The natural coaction $\rho_{(X,o)}$ in \eqref{eq:univcoa-2} 
  is universal. In other words,
  given a discrete $E^*$-module $N_*$,
  any   transformation of degree $0$
  $$\theta_{(X,o)}:E^*(X,o)\hat{}\lra E^*(X,o)\hat{\otimes} N_*,$$
  which is natural between functors from $\mathbf{Stab}^*$ to $\mathbf{GAb}^*$,
  is induced from $\rho_{(X,o)}$ by a unique homomorphism
  $f:E_*(E)\to N_*$ of left $E^*$-modules.
  That is, the following diagram in $\mathbf{GAb}$ is commutative:
$$\xymatrix{
  E^*(X,o)\ar[r]^-{\rho_{(X,o)}}\ar[rd]_{\theta_{(X,o)}}
  & E^*(X,o)\HOT E_*(E)\ar[d]^{1\HOT f}\\
& E^*(X,o)\HOT N_*.
}$$
We call $\theta_{(X,o)}$ a {\em stable coaction} over $N_*$.
The set of stable coactions over $N_*$ is isomorphic to
$\hom_{E_*\text{-mod.}}(E_*(E),N_*)=\mathbf{Mod}^0(E_*(E),N_*)$,
the homomorphisms of left $E_*$-modules.
  \end{enumerate}
\end{thm}
\begin{rem}
  \begin{enumerate}
  \item 
    The categories $\mathbf{Stab}^*$ and $\mathbf{GAb}^*$ are graded.
A stable coaction  $\theta$ over $N_*$ implies that
  the following diagram is commutative:
  \begin{equation}\label{eq:Estacoact}
    \xymatrix{
      E^*(X,o)\ar[r]^-{\theta_{(X,o)}}\ar[d]^{\Sigma}_{\cong}
      & E^*(X,o)\HOT N_*\ar[d]^{\Sigma \OT 1_{N_*}}_{\cong}\\
      E^{*}(\Sigma X,o)\ar[r]^-{\theta_{(\Sigma X,o)}}&
      E^{*}(\Sigma X,o)\HOT N_*,
      }
    \end{equation}
    where 
    $\Sigma $ is the suspension isomorphism \eqref{eq:susE}. 
   The vertical maps are of degree $1$.
  \item 
    Since a stable coaction  $\theta_{(X,o)}$ has the action scheme
    $$\theta_{(X,o)}:E2^*(X,o)\hat{}\lra E1^*(X,o)\hat{\otimes}_1 N1_*,$$
the coaction  $\theta_{(X,o)}$ does not preserve the action of $E2^*$.
    Hence $\theta_{(X,o)}$ is  a morphism  in $ \mathbf{GAb}^*$.
    In general, we can not define the composition of stable coactions 
    $$(\theta_X'\OT 1)\circ \theta_X:E^*(X,o)\hat{\ }
    \lra E^{*}( X,o)\HOT N_* \HOT N_*$$
    unlike the universal  stable coaction
    $\rho_X:E^*(X,o)\hat{\ }\to E^*(X,o)\HOT E_*(E) $.
  \end{enumerate}
\end{rem}
The universal stable coaction $\rho_X:E^*(X)\hat{\ }\to E^*(X)\HOT E_*(E)$
is compatible with cross product.
In other words,  we have the commutative diagram   of \cite[(12.3)]{board}
\begin{equation}\label{eq:comalg}
  \xymatrix{
    E2^*(X){\otimes_2} E2^*(Y)\ar[r]^-{\rho_X \otimes \rho_Y}\ar[dd]_{\times}
    &(E1^*(X)\hat{\otimes_1}E1_*(E2))\otimes_2 (E3^*(Y)\hat{\otimes_3}E3_*(E2))
    \ar[d]\\
    &(E1^*(X)\hat{\otimes}_{\Bbb{Z}} E3^*(Y))\hat{\otimes}_{E1_*\OT E3_*}
    (E1_*(E2){\otimes}_2E3_*(E2))
    \ar[d]^{\times \otimes \phi}\\
    E2^*(X\times Y)\ar[r]^{\rho_{X\times Y}} &E1^*(X\times Y)\hat{\otimes_1} E1_*(E2)
,}
\end{equation}
where $\times$ is the cross product and
the unlabeled map interchanges the second and the third factors.
Roughly speaking, the commutative diagram corresponds to
the dual of the Cartan formula $E^*(E)\to E^*(E)\HOT E^*(E)$.
For $x\in E^*(X)$ and $y\in E^*(Y)$,
we consider the action \eqref{eq:univaction} of $r\in E^*(E)$ on $x\times y$.
We  have a Cartan formula
$$ r(x\times y)=\sum_{\alpha}\pm r_{\alpha}'x\times r_{\alpha}''y$$
for suitable operations $r_{\alpha}'$ and $r_{\alpha}''$.
Here $\sum_{\alpha}$ may be infinite sum.
Then the multiplication $\phi:E\wedge E\to E$ on $E$ induces
$$ \phi^*(r)=\sum_{\alpha}r_{\alpha}'\times r_{\alpha}''\in E^*(E)\HOT E^*(E)\cong E^*(E\wedge E).$$
The Cartan formula converts to the commutative diagram \eqref{eq:comalg}
via the partial dualization \eqref{eq:univcoa-1}.
See \cite[Section 12]{board} for more details.
A similar commutative diagram
holds for smash product of spectra.

In \cite{board}, Boardman  
did not explain   multiplicative operations 
$\gamma_X:E^*(X)\to E^*(X)\HOT R_* $,
which is a generalization of operations with the commutative diagram
\eqref{eq:comalg}.
We give the definition, which is similar to Definition \ref{defi:mulop0}.
Recall that $\mathbf{GAb}$ is the category of graded abelian groups with
homomorphisms of degree 0.
\begin{defi}\label{defi:mulop}
  Let $R_*$ be a discrete $E^*$-module, and
  $E^*(-)\hat{\ }$ and $E^*(-)\HOT R_*$
  the  functors from $\mathbf{Ho}$ (or $\mathbf{Stab}$) to $\mathbf{GAb}$.
  A natural morphism
  $$\gamma_X:E^*(X)\hat{}\lra E^*(X)\HOT R_*$$
  in  $\mathbf{GAb}$ is called
  a  {\em multiplicative operation} over $R_*$
  if it satisfies the following conditions:
  \begin{enumerate}
  \item 
  The  diagram is commutative:
\begin{equation*}
  \xymatrix{
     E^*(X){\otimes}_{\Bbb{Z}} E^*(Y)\ar[r]^-{\gamma_X \otimes \gamma_Y}\ar[dd]_{\times}
    &(E^*(X)\hat{\otimes}R_*)\otimes_{\Bbb{Z}} (E^*(Y)\hat{\otimes}R_*)
    \ar[d]\\
    &(E^*(X){\otimes}_{\Bbb{Z}} E^*(Y))\hat{\otimes}_{E_* \OT E_*}
    (R_*{\otimes}_{\Bbb{Z}} R_*)
    \ar[d]^{\times \otimes \phi}\\
    E^*(X\times Y)\ar[r]^{\gamma_{X\times Y}} &E^*(X\times Y)\hat{\otimes} R_*
\ ;}
\end{equation*}
\item\label{6.3-2} $\gamma_{S^k}(u_k)=u_k \OT 1$ for $k=0,1$.
  Here $u_k\in E^k(S^k,o)$ is the canonical generator of  \eqref{eq:canogen}.
\end{enumerate}
We denote the set of multiplicative operations over $R_*$
by  $\op{Op}_E (R_*)$.
\end{defi} 
\begin{rem}
  By  substituting the sphere spectrum $(S^1,o)$ for $X$ in the commutative diagram of Condition (i),  Condition (ii) implies that
  a multiplicative operation $\gamma$ is compatible
with the suspension isomorphism such  as \eqref{eq:Estacoact}.
Therefore, any multiplicative operation  $\gamma $ over $R_*$ is
a stable coaction over $R_*$.
\end{rem}
The universal coaction $\rho$ in \eqref{eq:univcoa-1} or \eqref{eq:univcoa-2}
is a multiplicative operation over $E_*(E)$
since we have the commutative diagram \eqref{eq:comalg}.
That is to say, we have $\rho\in \op{Op}_E(E_*(E))$.
For any space $X$, the functors $E^*(X)\hat{\ }$ and $E^*(X)\HOT E_*(E)$
have  cup product, and are therefore  objects of $\mathbf{FAlg}$.

\begin{defi}[\cite{board} Definition 12.7.]\label{defi:comalg}
  Let $M$ be a stable $E$-cohomology  comodule.
  We call $M$ a {\em  stable (E-cohomology) comodule algebra}
  if $M$ is an object of $\mathbf{FAlg}$ and its coaction $\rho_M$
  is a morphism in $\mathbf{FAlg}$. That is to say,
  the following diagram is commutative:
\begin{equation}\label{eq:comalgM}
  \xymatrix{
    M{\otimes} M\ar[rr]^-{\rho_M \otimes \rho_M}\ar[dd]_{\mu}
    &&(M\hat{\otimes}E_*(E))\otimes (M\hat{\otimes}E_*(E))
    \ar[d]\\
    &&(M\hat{\otimes} M)\hat{\otimes}(E_*(E){\otimes}E_*(E))
    \ar[d]^{\times \otimes \phi}\\
    M\ar[rr]^{\rho_{M}} &&M\hat{\otimes} E_*(E)
,}
\end{equation}
where $\mu$ is the multiplication on $M$ and
the unlabeled map interchanges the second and the third factors.
\end{defi}
From the commutative diagram \eqref{eq:comalg},
we obtain (\ref{6.5-1}) of the following theorem.
\begin{thm}[\cite{board} Theorem 12.8.]\label{thm:univcomalg}
  \begin{enumerate}
\renewcommand{\theenumi}{\arabic{enumi}}
\renewcommand{\labelenumi}{(\theenumi)}
  \item\label{6.5-1} For any space $X$,
    the  coaction $\rho_X$ in \eqref{eq:univcoa-1}
    is a morphism in $\mathbf{FAlg}$, and 
     makes $E^*(X)\hat{}$ a stable comodule algebra.
  \item \label{6.5-2}
    The multiplicative operation  $\rho_{(X,o)}$ in \eqref{eq:univcoa-2}   is universal.
    That is to say,  for a discrete commutative $E^*$-algebra $R_*$,
  any multiplicative operation on $\mathbf{Stab}$ over $R_*$
  $$\gamma_{(X,o)}:E^*(X,o)\hat{\ }\lra E^*(X,o)\HOT R_*$$
  is induced from 
  a unique $E^*$-algebra homomorphism $\bar{\gamma}:E_*(E)\to R_*$ as follows:
$$\xymatrix{
  E^*(X,o)\hat{\ }\ar[r]^-{\rho_{(X,o)}}\ar[rd]_{\gamma_{(X,o)}}
  & E^*(X,o)\HOT E_*(E)\ar[d]^{1\HOT \bar{\gamma}}\\
& E^*(X,o)\HOT R_*.
} $$
Consequently,  we obtain the natural isomorphism of sets
\begin{equation}\label{eq:isoset}
\op{Op}_E(R_*)\cong\hom_{E^*\text{-alg}}(E_*(E),R_*)
\quad (\gamma_{(X,o)}\mapsto \bar{\gamma}).
\end{equation}

\end{enumerate}
\end{thm}
The completed cohomology  $E^*(-)\hat{\ }$ with $\rho$
is not only a stable comodule, but also a stable comodule algebra.
The structure induces the isomorphism \eqref{eq:isoset}.
In other words, the algebra $E_*(E)$ corepresents the functor $\op{Op}_E(-)$
of multiplicative operations.

From the Remark after Theorem \ref{thm:univcom},
the composition of
stable coactions $\theta_X,\ \theta_X'$ over $R_*$
\begin{align}\label{eq:iterateE}
(\theta'_X\OT 1)\circ \theta_X:E^*(X)\hat{\ }\lra E^*(X)\HOT R_*\HOT R_*  
\end{align}
is not available.
Since the universal coaction  $\rho_X$ is a  homomorphism of $E^*$-modules,
 the completed cohomology $E^*(X)\hat{\ }$ has the  composition 
$$(\rho_X\OT 1)\circ \rho_X:E^*(X)\hat{\ }\lra E^*(X)\HOT E_*(E)\HOT E_*(E)$$
in Definition \ref{defi:stacom}.
If  $H=H\Bbb{Z}/2$ or $ H\Bbb{Z}/p$,
then any multiplicative operation
$$\gamma_X:H^*(X)\hat{\ }\lra H^*(X)\HOT R_*$$
is a homomorphism of $\Bbb{Z}/2$ or $\Bbb{Z}/p$-modules.
In these cases, we can define the composition \eqref{eq:iterateE},
which induces the multiplication on $\op{Op}_H(R_*)$ in  \cite{I,I2}.
We improve Definition \ref{defi:mulop} and Theorem \ref{thm:univcomalg}
as follows.
\begin{thm}\label{prop:groupofmulop}
Suppose that  $E=H\Bbb{Z}/{2}$, $H\Bbb{Z}/{p}$ or $K(n)$.
   \begin{enumerate}
 \renewcommand{\theenumi}{\arabic{enumi}}
 \renewcommand{\labelenumi}{(\theenumi)}
\item\label{num:6.7-1}
  For any space $X$, a multiplicative operation
  $\gamma_X:E^*(X)\to E^*(X)\HOT R_*$ 
  is a  homomorphism of $E^*$-modules.
  Hence, we have the commutative diagram
  \begin{equation}\label{eq:comalgR}
  \xymatrix{
    E^*(X){\otimes} E^*(Y)\ar[r]^-{\gamma_X \otimes \gamma_Y}\ar[dd]_{\times}
    &(E^*(X)\hat{\otimes}R_*)\otimes (E^*(Y)\hat{\otimes}R_*)
    \ar[d]\\
    &(E^*(X)\hat{\otimes} E^*(Y))\hat{\otimes}(R_*{\otimes}R_*)
    \ar[d]^{\times \otimes \phi}\\
    E^*(X\times Y)\ar[r]^{\gamma_{X\times Y}} &E^*(X\times Y)\hat{\otimes} R_*,
}
\end{equation}
which is a refinement of Condition (i) in Definition \ref{defi:mulop} in that
the tensor products  are taken over $E_*$. 
Therefore $\gamma_X$ is a morphism in $\mathbf{FAlg}$,
which is analogous to Definition \ref{defi:comalg}.
\item \label{num:6.7-2}
  For
$\gamma,\gamma'\in \op{Op}_E(R_*)$,
we define  the multiplication $\gamma\cdot \gamma'$
  by the  composition 
  \begin{equation}
  \label{eq:prod-E}
\gamma\cdot \gamma':E^*(-)\xra{\gamma'}E^*(-)\HOT R_*
\xra{\gamma\OT 1}E^*(-)\HOT R_*\HOT R_*
\xra{1\OT \phi} E^*(-)\HOT R_*.
  \end{equation}
Then the identity element
$e\in \op{Op}_E(R_*)$ satisfies $e_{(X,o)}(x)=x\OT 1$ for any $x\in E^*(X,o)$.
The inverse element $\gamma^{-1}$ is given as the composition
$$ \gamma^{-1}_{(X,o)}:E^{*}(X,o)\xra{\rho_{(X,o)}}
E^*(X,o)\HOT E_*(E)\xra{1\OT \chi}E^*(X,o)\HOT E_*(E)
\overset{1\OT \bar{\gamma}}{\lra}E^*(X,o)\HOT R_*, $$
where $\chi$ is the conjugation of $E_*(E)$ and
$\bar{\gamma}$ is an  $E^*$-algebra homomorphism   corresponding to
$\gamma $ in Theorem \ref{thm:univcomalg}.
Thus $\op{Op}_E(-)$ is a functor from $\mathbf{Alg}_E$ to $ \mathbf{Gp}$.
Furthermore,
the natural transformation
\eqref{eq:isoset}
is an isomorphism of groups.
Here  the group structure of $\hom_{E^*\text{-alg}}(E_*(E),R_*)$
derives from the comultiplication $\psi $ on $E_*(E)$.
\end{enumerate}

\end{thm}
\begin{proof}
(1)\quad    Since $E_*(E)$ is a Hopf algebra,
    the left and right $E^*$-actions on $E_*(E)$ coincide. 
    Therefore, the universal coaction $\rho_X$ has the action scheme
    $\rho_X:E1^*(X)\to E1^*(X)\HOT_1 E1_*(E1)$.
    Any multiplicative operation $\gamma_X$ is a homomorphism of
    $E^*$-modules, and the commutative diagram in Definition \ref{defi:mulop}
    yields the commutative diagram \eqref{eq:comalgR}.
    This induces that $\gamma_X$ is a morphism in $\mathbf{FAlg}$. \\
(2)\quad     
    Since $\gamma_{(X,o)},\gamma_{(X,o)}'$ are  $E^*$-module  homomorphisms,
    we can define  $\gamma\cdot \gamma'.$
  Obviously, we have $\gamma\cdot e=e\cdot \gamma =\gamma$.
  By Theorem \ref{thm:univcomalg} (\ref{6.5-2}),
  the multiplication $\gamma\cdot \gamma'$
  is expressed as
\begin{equation}\label{eq:multopE}
  \begin{aligned}
    \gamma\cdot \gamma':E^*(X,o)\hat{}\xrightarrow{\rho_{(X,o)}}
    E^*(X,o)\HOT E_*(E)&
    \xra{\rho_{(X,o)}\OT 1}E^*(X,o)\HOT E_*(E)\HOT E_*(E)
\\&    \xra{1\OT \bar{\gamma}\OT \bar{\gamma}'}
    E^*(X,o)\HOT R_*\HOT R_*
    \xra{1\OT \phi}    E^*(X,o)\HOT R_*,
  \end{aligned}
\end{equation}
where $\bar{\gamma}$ and $\bar{\gamma}'$ are $E^*$-algebra homomorphisms
    corresponding to $\gamma$ and $\gamma'$, respectively.
    By Definition \ref{defi:stacom},
    we have
    $$
    \xymatrix{
(\rho_{(X,o)}\OT 1)\circ\rho_{(X,o)}=(1\OT \psi)\circ \rho_{(X,o)}:
E^*(X,o)\hat{}\ar[r]^-{\rho_{(X,o)}} &E^*(X,o)\HOT E_*(E)
\ar@<0.5ex>[r]^-{\rho_{(X,o)}\OT 1} \ar@<-0.5ex>[r]_-{1\OT \psi} 
& E^*(X,o)\HOT E_*(E)\HOT E_*(E).
      }
      $$
  It follows from \eqref{eq:multopE} that
  the multiplication $\gamma\cdot \gamma '$ on $\op{Op}_E(R_*)$ corresponds to
  the homomorphism 
$$\hfill \overline{\gamma\cdot \gamma'}=\phi \circ(\bar{\gamma}\OT\bar{\gamma}')\circ\psi:E_*(E)\overset{\psi}{\lra}E_*(E)\OT E_*(E)
\xrightarrow{\bar{\gamma} \OT \bar{\gamma}'} R_*\OT R_*
\overset{\phi}{\lra} R_*.
$$
Therefore the isomorphism  $\op{Op}_E(R_*)\cong\hom_{E^*\text{-alg}}(E_*(E),R_*)$ in Theorem \ref{thm:univcomalg} (\ref{6.5-2})
preserves  multiplication, and it is an isomorphism of groups.
  It is easily checked that the conjugation on $E_*(E)$ induces
  the inverse element $\gamma^{-1}_X$.
\end{proof}
In the case $E=K(n)$, we see that $\op{Op}_{K(n)}(-)$ is a group-valued functor
which is corepresented by the Hopf algebra $K(n)_*(K(n))$.
It is well known 
that $K(n)_*(K(n))$ is a Hopf algebra.
In \cite[Section 14, Example $K(n)$]{board},
Boardman  illustrated it by using stable comodule algebras.

\section{Relations between  $C(-)$ and
  $\op{Op}(-)$}
  \label{sec:isom}
From now on, we set $E=K(n)$, and
abbreviate the multiplicative operations $\op{Op}_{K(n)}(-)$
of Definition \ref{defi:mulop}  to $\op{Op}(-)$.
  In this section,
  we construct a natural transformation
  $$\kappa:\op{Op}(-)\lra C(-),$$ 
  where $C(-)$ is the functor of Definition \ref{defi:fp}.
 Any $K(n)_*$-algebra $R_*$ can be treated as a complete Hausdorff
 $K(n)_*$-algebra with trivial filtration.
 Hence the category $\mathbf{Alg}$ of $K(n)^*$-algebras coincides with
 $\mathbf{Alg}_{K(n)}$.
  Both $\op{Op}(-)$ and $C(-)$ are functors  from $\mathbf{Alg}$
  to  $\mathbf{Gp}$, the category of groups.
  By Yoneda's Lemma,  Theorem \ref{thm:fp} and
  Theorem \ref{prop:groupofmulop} (\ref{num:6.7-2}),
  the transformation $\kappa$ induces  the Hopf algebra homomorphism
  $\kappa^*:C_*\to K(n)_*(K(n))$.
   Using the algebra structure of $K(n)_*(K(n))$ in \cite{W2,Y1}, 
  we see that $\kappa^*$ is an isomorphism.
  
  To obtain $\kappa$, we need three natural transformations
  $$\kappa_1:\op{Op}(-)\lra \aut_{H_n}(-),\
  \kappa_2:\op{Op}(-)\lra \aut_{\bar{g}_a}(-),\
    \kappa_3:\op{Op}(-)\lra \aut_{\bar{G}_a}(-) $$
  such that the following diagram is commutative:
\begin{equation}
\xymatrix{
  \op{Op}(-)\ar[r]^{\kappa_1}\ar[d]_{\kappa_3}\ar[rd]^{\kappa_2} & \aut_{H_n}(-)\ar[d]^{\alpha}\\
  \aut_{\bar{G}_a}(-)\ar[r]^{\beta}& \aut_{\bar{g}_a}(-).
}
\end{equation}
Let $\gamma_X:K(n)^*(X)\to K(n)^*(X)\HOT R_*$ be a multiplicative operation,
$\CP^{\If}$  the infinite complex projective space,
$\CP^{p^n-1}$ the $(p^n-1)$-dimensional complex projective space,
and $\BZN$ the $(2p^n-1)$-skeleton of the classifying space $B\Bbb{Z}/p$.
Substituting $X=\CP^{\If},\CP^{p^n-1},\BZN$ into $\gamma_X$, 
we have the natural transformations $\kappa_1, \kappa_2, \kappa_3$.

Before constituting the natural transformations,
we fix generators of $K(n)^*(\CP^{\If}),\ K(n)^*(\CPN)$ and $K(n)^*(\BZN)$.
Let $x$ be the canonical generator of $K(n)^2(\CP^{\If})$ which
satisfies $i^*(x)=u_2  \in K(n)^2(S^2,o)$. 
Here $i:S^2\to \CP^{\If}$ is the usual inclusion,
and  $u_k$ is the  generator of $K(n)^k(S^k,o)$ in \eqref{eq:canogen}.
We have
\begin{equation}
  \label{eq:KNCP}
K(n)^*(\CP^{\If})\cong K(n)^*[[x]].  
\end{equation}
The inclusions $j_1:\CPN\to \CP^{\If}$ and
$j_2:\BZN\to \CPN$ induce
the elements $j_1^*(x)\in K(n)^2(\CPN)$
and $j_2^*\circ j_1^*(x)\in K(n)^2(\BZN)$.
Using the same notation, we denote the elements by $x$.
We obtain
\begin{equation}
  \label{eq:KNCPN}
K(n)^*(\CPN)\cong K(n)^*[x]/(x^{p^n}).   
\end{equation}
The usual inclusion $i:S^1\to \BZN$
leads to the element $\ee$ of $K(n)^1(\BZN)$ such that
$i^*(\ee )=u_1\in K(n)^1(S^1,o)$, and we express the cohomology of $\BZN$ as
\begin{equation}\label{eq:KNBZN}
K(n)^*(\BZN)\cong E(\ee )\OT K(n)^*[x]/(x^{p^n}).  
\end{equation}

Firstly, we study the case   $X=\CP^{\If}$.
The multiplication $$\mu:\CP^{\infty}\times \CP^{\infty}\to \CP^{\infty}$$
induces the comultiplication on $K(n)^*(\CP^{\infty})\cong K(n)^*[[x]]$.
It is well known \cite[Theorem 3.11(b)]{RW} that the element $\mu^*(x)$ is the $p$-typical formal group law 
which  is characterized by the $p$-series $[p](x)=v_nx^{p^n}$.
This means that $\mu^*(x)$ coincides with the Honda formal group law $H_n$
over $K(n)_*$. 
\begin{prop}\label{prop:optoHn}
  Given  $\gamma_X\in \op{Op}(R_*)$ and the  generator
  $x\in K(n)^*(\CP^{\If})$ of \eqref{eq:KNCP},
  the power series
  $$\gamma_{\CP^{\If}}(x)\in K(n)^*(\CP^{\If})\HOT R_*\cong R_*[[x]]$$
  is a strict automorphism of $H_n$ over $R_*$, that is,
  $\gamma_{\CP^{\If}}(x)\in \aut_{H_n}(R_*)$.
  Moreover, this induces the natural homomorphism of groups
\begin{equation}
  \label{eq:nt-hn}
  \kappa_1:\op{Op}(R_*)\cong \hom(K(n)_*(K(n)),R_*)\lra \aut_{H_n}(R_*)
  \cong \hom({\Sigma}(n),R_*)\quad (\gamma \mapsto \gamma_{\CP^{\If}}(x)).
\end{equation}
\end{prop}
\begin{proof}
We consider the commutative diagram
\begin{equation}\label{eq:commdia1}
  \xymatrix{
    K(n)^*(\CP^{\infty})
    \ar[rr]^{\gamma_{\CP^{\infty}}}
     \ar[d]_{\mu^*}&
     &K(n)^*(\CP^{\infty})\hat{\otimes}R_*
     \ar[d]^{\mu^{*}\otimes 1}\\
     K(n)^*(\CP^{\infty} \times \CP^{\infty})
     \ar[rr]^{\gamma_{{\CP^{\If}}\times \CP^{\If}}}
     & &
     K(n)^*(\CP^{\infty}\times \CP^{\infty})\hat{\otimes}R_*
     \\
      & &(K(n)^*(\CP^{\infty}))\hat{\otimes} (K(n)^*(\CP^{\infty}))
      \hat{\otimes} (R_*\hat{\otimes} R_*) 
      \ar[u]_{\times\otimes \phi}
      \\
      K(n)^*(\CP^{\infty})\hat{\otimes}
      K(n)^*(\CP^{\infty})\ar[uu]^{\cong}_{\times}
      \ar[rr]^-{\gamma_{\CP^{\If}}\otimes \gamma_{\CP^{\If}}}
      & &(K(n)^*(\CP^{\infty})\hat{\otimes}R_*) \hat{\otimes} (K(n)^*(\CP^{\infty})\hat{\otimes}R_*).
      \ar[u]
    }
  \end{equation}
Since a multiplicative operation $\gamma_X$ is natural in $X$,  
the upper diagram  is commutative.
It follows from the commutative diagram \eqref{eq:comalgR}
in Theorem \ref{prop:groupofmulop} that the lower diagram is commutative.
Under  the upper-right vertical map
$$\mu^*\OT 1:     K(n)^*(\CP^{\infty})\hat{\otimes}R_*\cong R_*[[x]]\lra
K(n)^*(\CP^{\infty}\times \CP^{\infty})\hat{\otimes}R_*\cong
R_*[[x\times 1,1\times x]],$$
the image $(\mu^*\OT 1)(x)$ is the Honda formal group law
$H_n(x\times 1,1\times x)$ over $R_*$.
Hence the commutative diagram \eqref{eq:commdia1}
implies that $\gamma_{\CP^{\If}} (x)$ is an automorphism of $H_n$ over $R_*$.
For $\gamma_{\CP^{\If}} (x)=\sum_{i\geq 1} a_ix^i\in R_*[[x]]$,
the standard inclusion $i:S^2\to \CP^{\If}$ and the definition of $x$  induce
$$i^*\left(\gamma_{\CP^{\If}} (x)\right)=
i^*\left(\sum_{i\geq 1} a_ix^i\right)= a_1u_2.  $$
By Definition \ref{defi:mulop} (\ref{6.3-2}), we have $a_1=1$.
Therefore $\gamma_{\CP^{\If}}(x)$ is a strict automorphism of $H_n$ over $R_*$.
Assigning $\gamma_{\CP^{\If}}(x)$ to a multiplicative operation $\gamma$,
we obtain the natural transformation $\kappa_1$.
For elements $\gamma,\gamma'$ of $\op{Op}(R_*)$,
it follows from  Definition \ref{defi:Hn} (\ref{2.1-4})
and   Theorem \ref{prop:groupofmulop} (\ref{num:6.7-2}) 
that
\begin{align*}
  \kappa_1(\gamma\cdot \gamma')&=
              ((1\otimes \phi)\circ(\gamma_{CP^{\If}}\otimes 1)\circ
  \gamma_{CP^{\If}}')(x)
       =\gamma_{CP^{\If}}'(\gamma_{CP^{\If}}(x))\\
  &=\gamma_{CP^{\If}}'\circ    \gamma_{CP^{\If}}(x)=
    \gamma_{CP^{\If}}(x)\cdot \gamma_{CP^{\If}}'(x)\quad
  \in R_*[[x]]
  \cong K(n)^*(\CP^{\If})\hat{\otimes}R_*.
\end{align*}
Here $\gamma_{CP^{\If}}'(\gamma_{CP^{\If}}(x))=\gamma_{CP^{\If}}'\circ    \gamma_{CP^{\If}}(x)$ is the composition in $R_*[[x]]$, and
$\gamma_{CP^{\If}}(x)\cdot \gamma_{CP^{\If}}'(x)$ is the multiplication
on $\aut_{H_n}(R_*)$.
Therefore $\kappa_1$ is a group homomorphism.
\end{proof}

Secondly, we consider the case $X=\CP^{p^n-1}$,
whose cohomology is $K(n)^*(\CP^{p^n-1})\cong K(n)^*[x]/(x^{p^n})$.
Though $\CPN$ is not an $H$-space,
we have partial multiplications
$$\mu_k:\CP^{k}\times \CP^{p^n-1-k}\lra \CP^{p^n-1}, $$
which are defined by the restriction maps of $\mu$.
For $Y=\bigcup_{k=0}^{p^n-1} \CP^{k}\times \CP^{p^n-1-k}$,
we obtain the multiplication
$ \bar{\mu}:Y\lra \CP^{p^n-1}$.
This induces
$$\bar{\mu}^*:K(n)^*(\CPN)\cong K(n)^*[x]/(x^{p^n})\lra K(n)^*[x\times 1,1\times x]/
(x\times 1,1\times x)^{p^n}\cong K(n)^*(Y).$$
The element $\bar{\mu}^*(x)$ is homogeneous of degree 2.
By comparison of degrees, we have $\bar{\mu}^*(x)=x\times 1+1\times x$, which coincides with
the additive group law chunk $\bar{g}_a(x\times 1,1\times x)$
of size $p^n-1$.
\begin{prop}\label{prop:optoga}
  For  $\gamma_X\in \op{Op}(R_*)$ and the  generator
  $x\in K(n)^*(\CPN)\cong K(n)_*[x]/(x^{p^n})$ in \eqref{eq:KNCPN},
  the element
  $$\gamma_{\CPN}(x)\in K(n)^*(\CPN)\HOT R_*\cong R_*[x]/(x^{p^n})$$
  is a strict automorphism of $\bar{g}_a$ over $R_*$, that is,
  $\gamma_{\CPN}(x)\in \aut_{\bar{g}_a}(R_*)$.
 It induces the natural homomorphism of groups
  $$\kappa_2: \op{Op}(R_*)\cong \hom(K(n)_*(K(n)),R_*)\lra \aut_{\bar{g}_a}(R_*)
   \cong \hom(A_*,R_*)\quad (\gamma\mapsto \gamma_{\CPN}(x)).
   $$
\end{prop}
\begin{proof}
Since $\gamma_X$  is a multiplicative operation over $R_*$,
we have the commutative diagram 
\begin{equation*}
  \xymatrix{
        K(n)^*(\CP^{p^{n-1}})
    \ar[rr]^{\gamma_{\CP^{p^{n-1}}}}
     \ar[d]_{\mu_k^*}
    & &K(n)^*(\CP^{p^{n-1}})\hat{\otimes}R_* \ar[d]^{\mu_k^{*}\otimes 1} 
     \\
         K(n)^*(\CP^{k}\times \CP^{p^n-1-k})
         \ar[rr]^{\gamma_{\CP^k\times\CP^{p^n-1-k}}}  &   & 
         K(n)^*(\CP^{k}\times \CP^{p^n-1-k})\hat{\otimes}R_*
         \\
      && (K(n)^*(\CP^k)\hat{\otimes}K(n)^*(\CP^{p^n-1-k}))
      {\HOT} (R_*\hat{\otimes}R_*)
      \ar[u]_{\times\OT\phi}
      \\
      K(n)^*(\CP^{k}){\otimes}
      K(n)^*(\CP^{p^n-1-k})\ar[uu]_{\times}^{\cong}
               \ar[rr]^-{\gamma_{\CP^k}\otimes \gamma_{\CP^{p^n-1-k}}}
      &&(K(n)^*(\CP^k)\hat{\otimes}R_*)
      {\otimes} (K(n)^*(\CP^{p^n-1-k})\hat{\otimes}R_*)
     \ar[u]
    }
  \end{equation*}
  The diagram induces
  \begin{align*}
    \gamma_{\CP^k\times\CP^{p^n-1-k}}(x\times 1+1\times x)
     &=\gamma_{\CP^k}(x)\times 1+1\times \gamma_{\CP^{p^n-1-k}}(x).
  \end{align*}
  Using the multiplication $\bar{\mu}$, we see
  $$\gamma_{Y}(x\times 1 +1\times x)
    =\gamma_{\CPN}(x)\times 1+1\times \gamma_{\CPN}(x).
   $$
   By the equality, Definition \ref{defi:mulop} (\ref{6.3-2}) and
   the definition  of $x$, we have $\gamma_{\CPN}(x)\in \aut_{\bar{g}_a}(R_*)$.
  In a similar way to the proof of Proposition \ref{prop:optoHn},
  we see that $\kappa_2$ is a group homomorphism.
\end{proof}

Thirdly, we investigate the case  $X=\BZN$.
In a similar way to $X=\CPN$, we obtain the partial multiplications
$$m_k:\BZ^{[k+1]}\times \BZ^{[2p^n-k-2]}\lra X=\BZN. $$
For $Z=\bigcup_{k=0}^{2p^n-2}\BZ^{[k+1]}\times \BZ^{[2p^n-k-2]}$,
the multiplication $m:Z\to \BZN$ induces
$$m^*:E(\ee )\OT K(n)^*[x]/(x^{p^n})\lra E(\ee\times 1,1\times \ee )\OT
K(n)^*[x\times 1,1\times x]/(x\times 1,1\times x)^{p^n}
.$$
By comparison of degrees, we have $m^*(x)=1\times x+x\times 1$, and
$m^*(\ee )=\ee\times 1+1\times \ee  +\lambda v_n\ee x^{p^n-1}$
for some $\lambda\in \FP $. 
Since $B\Bbb{Z}/p$ is a homotopy associative $H$-space,
the pair $(m^*(\ee ),m^*(x))$ is a graded 2-dimensional formal group law chunk
in Definition \ref{defi:Ga2}(\ref{defi:Ga2-1}).
By associativity, we can easily check  $m^*(\ee )=\ee\times 1+1\times \ee $.
Hence the formal group law chunk  $(m^*(\ee ),m^*(x))$ coincides
with $\bar{G}_a$
defined in Section \ref{sec:2fgl}.
We can prove the following proposition in a similar way to Proposition
\ref{prop:optoga}.

\begin{prop}\label{prop:optoGa}
  Let $\gamma_X\in \op{Op}(R_*)$ be a multiplicative operation over $R_*$,
  and $\ee$ and $ x$ the generators of  $K(n)^*(\BZN)$ in \eqref{eq:KNBZN}.
  The pair
  $$\bm{\gamma}(\ee,x)=(\gamma_{\BZN}(\ee),\gamma_{\BZN}(x))$$
  of elements of $K(n)^*(\BZN)\HOT R_*\cong E(\ee)\OT R_*[x]/(x^{p^n})$
  is a quasi-strict automorphism of $\bar{G}_a$ over $R_*$, that is,
  $\bm{\gamma}(\ee ,x) \in \aut_{\bar{G}_a}(R_*)$.
  Assigning $\bm{\gamma}$ to ${\gamma}_X \in \op{Op}(R_*)$,
  we have the natural homomorphism of groups
\begin{equation}
  \label{eq:eq:nt-Ga}
  \kappa_3:\op{Op}(R_*)\cong \hom(K(n)_*(K(n)),R_*)\lra \aut_{\bar{G}_a}(R_*)
  \cong \hom(B_*,R_*).
\end{equation}

\end{prop}
\begin{proof}
  In a similar way to Proposition \ref{prop:optoga},
  the pair $\gamma(\ee,x)$ satisfies 
  \begin{align*}
    &\gamma_Z(\ee \times 1+1\times \ee )=\gamma_{\BZN}(\ee )\times 1+ 1\times \gamma_{\BZN}(\ee)\\
    &\gamma_Z(x\times 1+1\times x)=\gamma_{\BZN}(x)\times 1+ 1\times \gamma_{\BZN}(x).
  \end{align*}
  That is, $\bm{\gamma}(\ee ,x)$ is an automorphism of $\bar{G}_a$.
  It follows from Definition \ref{defi:mulop} (\ref{6.3-2})
  and the definitions of $\ee$ and $x$
  that $\bm{\gamma}(\ee ,x)$ is a quasi-strict automorphism of $\bar{G}_a$.
\end{proof}

Lastly, we construct a Hopf algebra homomorphism $\kappa^*:C_*\to K(n)_*(K(n))$,
and show that $\kappa^*$ is an isomorphism.
\begin{thm}
  Let $C(-)$ be the functor of Definition \ref{defi:fp}.
  The natural transformations $\kappa_1,\kappa_2,\kappa_3$ induce
  the natural homomorphism of groups
$$\kappa:\op{Op}(-)\cong \hom_{K(n)\text{-alg}}(K(n)_*(K(n)),-)
\lra C(-)\cong \hom_{K(n)\text{-alg}}(C_*,-).$$
Therefore $\kappa$ yields
the corresponding homomorphism of Hopf algebras 
$$\kappa^*:C_*\lra K(n)_*(K(n)).$$

\end{thm}
\begin{proof}
It follows from the surjection
$$j_1^*:K(n)^*(\CP^{\If})\cong K(n)^*[[x]]\to K(n)_*(\CPN)\cong
K(n)^*[x]/(x^{p^n})\quad (x\mapsto x),$$
Proposition \ref{prop:Hntoga} and the map \eqref{eq:Hntoga} that
we have $\kappa_2=\alpha\circ\kappa_1$.
Via the injection
$$j_2^*:K(n)^*(\CPN)\cong K(n)^*[x]/(x^{p^n})\to K(n)^*(\BZN)\cong
E(u)\OT K(n)^*[x]/(x^{p^n}) \quad (x\mapsto x),
$$
we have
$\gamma_{\CPN}(x)=\gamma_{\BZN}(x) \in K(n)^*[x]/(x^{p^n})$.
By Proposition \ref{prop:Gatoga}, we obtain
 $\kappa_2=\beta\circ\kappa_3$.
The commutative diagram  
\begin{equation*}
  \label{eq:kappa}
\xymatrix{
  \op{Op}(R_*) \ar[rrd]^{\kappa_1}\ar[rdd]_{\kappa_3}\ar@{.>}[rd]_{\kappa} & & \\
  &C(R_*)\ar[r]\ar[d] & \aut_{H_n}(R_*)\ar[d]^{\alpha}\\
  &\aut_{\bar{G}_a}(R_*)\ar[r]^{\beta}& \aut_{\bar{g}_a}(R_*).
}  
\end{equation*}
induces the natural transformation $\kappa$ and the corresponding
 homomorphism $\kappa^*$.
\end{proof}
We recall  the following theorem, which
is due to W\"urgler \cite{W2} and Yagita \cite{Y1}.
\begin{thm}\cite[Theorem 14.32 (a)]{board}\label{thm:yagita}
  We express  $\rho_{\CP^{\If}}(x)$ and
  $\rho_{\BZN}(\ee )$ as
$$\rho_{\CP^{\If}}(x)=\sum_{i=1}^{\infty} x^i\otimes b_i,\
\rho_{\BZN}(\ee )=\sum_{i=0}^{p^n-1} x^i\otimes a_i+
\sum_{i=0}^{p^n-1} \ee x^i\otimes c_i,
$$
respectively.
The commutative $K(n)_*$-algebra $K(n)_*(K(n))$ has the generators
$a_{p^i}$ for $0\leq i<n$ and $b_{p^i}$ for $i>0$ subject the relations
$b_{p^i}^{p^n}=v_n^{p^i-1}b_{p^i}$.
\end{thm}
 The generators $a_{p^i}$ and $b_{p^i}$ correspond to 
 $\bar{t}_i$ and $\bar{\tau}_i$, the conjugation elements of
  $t_i,\ \tau_i\in K(n)_*(K(n))$.
  W\"urgler and Yagita  determined the algebra structure of $K(n)_*(K(n))$,
 and we obtain the algebra $C_*$ in Theorem \ref{thm:fp}.
 Moreover, we investigate the generators of $C_*$ in Theorems \ref{lem:repHn} and \ref{prop:repGa}. 
Theorem \ref{thm:yagita} implies the following main theorem.
\begin{thm}
  $\kappa^*:C_*\to K(n)_*(K(n))$ is an isomorphism of  Hopf algebras.
\end{thm}



\begin{thebibliography}{99}
\bibitem{adams}
  J.\ F.\ Adams,
\newblock{Stable homotopy and generalised homology.}
{Chicago Lectures in Mathematics. University of Chicago Press, Chicago, Ill.-London, 1974. x+373 pp.}
\bibitem{B}A.\ Baker,
\newblock{$A_{\If}$ structures on some spectra related to Morava $K$-theories.}
{Quart. J. Math. Oxford Ser. (2) 42 (1991), no. 168, 403–419. }
\bibitem{BW2}A.\ Baker  and U.\ W\"{u}rgler,
  \newblock{Liftings of formal groups and the Artinian completion
    of $v_n^{-1} BP$.}
Math. Proc. Cambridge Philos. Soc. 106 (1989), no. 3, 511–530. 
\bibitem{BW}A.\ Baker  and U.\ W\"{u}rgler,
  \newblock{Bockstein operations in Morava $K$-theories.}
  { Forum Math. 3 (1991), no. 6, 543–560. }
  \bibitem{board}
    J.\  M.\ Boardman,  
    \newblock{Stable operations in generalized cohomology.}
    \newblock{Handbook of algebraic topology, 585-686, North-Holland, Amsterdam, 1995.}
  \bibitem{C}
    H.\  Cartan,
    \newblock{Sur les groupes d'Eilenberg-Mac Lane. II.},
    Proc. Nat. Acad. Sci. U.S.A. 40 (1954), 704–707. 
  \bibitem{H}
    M.\  Hazewinkel,
    \newblock{Formal groups and applications.}
    Pure and Applied Mathematics, 78. Academic Press, Inc. [Harcourt Brace Jovanovich, Publishers], New York-London, 1978. xxii+573pp.
  \bibitem{Honda}
    T.\ Honda,
    \newblock{On the theory of commutative formal groups.}
    J. Math. Soc. Japan 22 (1970), 213–246.
  \bibitem{hop}
    M.\ Hopkins,
    \newblock{Complex oriented cohomology theories and the language of stacks,}
    Course notes for 18.917 at MIT, August 13,1999. Available from
    \url{https://www.semanticscholar.org/paper/COMPLEX-ORIENTED-COHOMOLOGY-THEORIES-AND-THE-OF-FOR/d213545f8c1debace8fb642925ca32516276d56f}
  \bibitem{I}
    M.\ Inoue,
    \newblock{ The Steenrod algebra and the automorphism group of additive formal group law.}
    J. Math. Kyoto Univ. 45 (2005), no. 1, 39–-55.
  \bibitem{I2}
     M.\ Inoue,
    {\newblock Odd primary Steenrod algebra, additive formal group laws, and modular invariants.}
    J. Math. Soc. Japan 58 (2006), no. 2, pp. 311-–332.
  \bibitem{L}
    L.\  Lomonaco,
    \newblock{The iterated total squaring operation.}
    { Proc. Amer. Math. Soc. 115 (1992), no. 4, 1149-–1155.}
      \bibitem{mil}
        J.\ Milnor,
        \newblock{The Steenrod algebra and its dual.}
        {Ann. of Math. (2) 67 (1958), 150–-171. }
      \bibitem{Mui1}H.\ M\`{u}i,
        \newblock{Modular invariant theory and cohomology algebras of symmetric groups.}
         J. Fac. Sci. Univ. Tokyo Sect. IA Math. 22 (1975), no. 3, 319--369.
       \bibitem{Mui2}H.\ M\`{u}i,
         \newblock{Dickson invariants and the Milnor basis of the Steenrod algebra,
           Topology and applications.}
         Colloq. Math. Soc. J´anos Bolyai 41 (1983), 345–-355.

  \bibitem{R}
    D.\ C.\ Ravenel,     
{Complex cobordism and stable homotopy groups of spheres,}
Pure and Applied Mathematics, 121. Academic Press, Inc., Orlando, FL, 1986. xx+413

\bibitem{RW}
  D.\ C.\ Ravenel,    W.\ S.\ Wilson,
  \newblock{The Hopf ring for complex cobordism.}
  {J. Pure Appl. Algebra9(1976/77/1977), no.3, 241–-280.}
\bibitem{serre}
     J.\ P.\ Serre,
     \newblock{Cohomologie modulo 2 des complexes d'Eilenberg-MacLane.}
     {  Comment. Math. Helv. 27 (1953), 198–-232.}
  \bibitem{steenrod}
    N.\ E.\ Steenrod and D.\ B.\ A.\ Epstein,
    \newblock{Cohomology Operations,}
    {Annals of Math.\ Studies No.\ 50,
     Princeton University Press, 1962}
 \bibitem{W2}
   U.\ W\"{u}rgler,
   \newblock{On the relation of Morava K-theories to Brown-Peterson homology. Topology and algebra,}
   {Monograph. Enseign. Math., 26, Univ. Genève, Geneva, 1978,269--280}
 \bibitem{W3}
   U.\ W\"{u}rgler,
   \newblock{Morava K-theories: a survey.}
   {Algebraic topology Poznań 1989, 111–138, Lecture Notes in Math., 1474, Springer, Berlin, 1991.}
 \bibitem{Y2}N.\ Yagita,
   \newblock{The exact functor theorem for $BP_* /I_n$-theory.}
    Proc. Japan Acad. 52 (1976), no. 1, 1–-3. 
  \bibitem{Y1}N.\ Yagita,
    \newblock{A topological note on the Adams spectral sequence based on Morava's K-theory.}
    Proc. Amer. Math. Soc. 72 (1978), no. 3, 613–-617.
    
  \bibitem{yam}
    A.\ Yamaguchi,
    \newblock{On excess filtration on the Steenrod algebra.}
    {Proceedings of the Nishida Fest (Kinosaki 2003), 423-–449, Geom. Topol. Monogr., 10, Geom. Topol. Publ., Coventry, 2007.}
  \end{thebibliography}
\end{document}